\documentclass[10pt]{article}
\usepackage{pifont}
\usepackage{amsmath}
\usepackage{amscd}
\usepackage{amsfonts}
\usepackage{latexsym}
\usepackage{amssymb}
\usepackage{stmaryrd}
\usepackage{overpic}
\usepackage{graphicx}
\usepackage{undertilde}
\usepackage{empheq}
\usepackage{subcaption}
\usepackage{empheq}
\usepackage{multirow}
\usepackage{enumitem}
\usepackage{lineno}
\usepackage[unicode,colorlinks,linkcolor=blue,hyperindex]{hyperref}

\usepackage{tikz}
\usetikzlibrary{er}

\newtheorem{theorem}{Theorem}[section]
\newtheorem{assumption}[theorem]{Assumption}
\newtheorem{corollary}[theorem]{Corollary}

\newtheorem{example}[theorem]{Example}
\newtheorem{lemma}[theorem]{Lemma}

\newtheorem{remark}[theorem]{Remark}
\newenvironment{proof}[1][Proof]{\textbf{#1.} }
{\ \rule{0.75em}{0.75em}\smallskip}

\numberwithin{equation}{section}

\newcommand{\vect}[1]{\boldsymbol{#1}} 

\begin{document}

\title{Uniform Stability and Error Analysis\\
for Some Discontinuous Galerkin Methods 
}
\author{Qingguo Hong and Jinchao Xu}
\date{}
\maketitle 

\begin{abstract}

In this paper,  we provide a number of new estimates on the stability and convergence of both hybrid discontinuous Galerkin (HDG) and weak Galerkin (WG) methods.  
By using the standard Brezzi theory on mixed methods, we carefully define appropriate norms for the various discretization variables and then establish that 
the stability and error estimates 
hold uniformly with respect to stabilization and discretization parameters.  
 As a result, by taking appropriate limit of the stabilization parameters, we show that the HDG method converges to 
a primal conforming method and the WG method converge to a mixed conforming method.
\end{abstract}



\section{Introduction} \label{sec:intro} 
 In the last few decades,
one variant of finite element method called the discontinuous Galerkin (DG) method
\cite{karniadakis2000discontinuous, arnold2002unified} has been developed to solve various 
differential equations due to their flexibility in
constructing feasible local shape-function spaces and the advantage of effectively capturing 
non-smooth or oscillatory solutions. Since DG methods use discontinuous space as trial space, 
the number of degrees of freedom is usually much higher than the standard conforming method. 
To reduce the number of globally coupled degrees of freedom of DG methods,
a hybrid DG (HDG) has been developed. The idea of hybrid methods can be tracked 
to the 1960s \cite{fraeijs1965displacement}.
A new hybridization approach in \cite{cockburn2004characterization} was put forward by Cockburn and Gopalakrishnan in 2004 and was successfully applied to a discontinuous Galerkin method in \cite{carrero2006hybridized}. 
Using the local discontinuous Galerkin (LDG) method to define the local solvers, a super-convergent LDG-hybridizable Galerkin method for second-order elliptic problems was designed in \cite{ cockburn2008superconvergent}. In 2009, a unified analysis for the hybridization of discontinuous Galerkin, mixed, and continuous Galerkin methods for second-order elliptic problems was presented in \cite{cockburn2009unified} by Cockburn, Gopalakrishnan, and Lazarov. A projection-based error analysis of HDG methods was presented in \cite{cockburn2010projection}, 
where a projection was constructed to obtain the $L^2$ error estimate for the potential and flux. 
However, the error estimate was dependent on the stabilization parameter. 
A projection-based analysis of the hybridized discontinuous Galerkin methods for convection-diffusion equations for semi-matching nonconforming meshes was presented in \cite{chen2014analysis}. An analysis for a hybridized discontinuous Galerkin 
method with reduced stabilization for second-order elliptic problem was given in \cite{oikawa2015hybridized}.



Based on a new concept, namely the weak gradient, introduced in \cite{wang2013weak}, 
Wang and Ye proposed a weak Galerkin (WG) method for 
elliptic equations. Similar to the concept introduced in \cite{wang2013weak}, Wang and 
Ye \cite{wang2014weak} introduced a concept 
called weak divergence. Based on the newly introduced concept, Wang and Ye \cite{wang2014weak} proposed and analyzed 
a WG method for the second-order elliptic equation formulated as a system of two first-order linear equations.
Then a similar idea was applied to Darcy-Stokes flow in \cite{chen2016weak}. A primal-dual WG 
finite element method for second-order elliptic equations in non-divergence form 
was presented in \cite{wang2017primal} and a further similar method was applied to Fokker-Planck
type equations in \cite{wang2017primalFP}. A bridge building the connection between the 
WG method and HDG method was shown in \cite{cockburn2016static}. A summary of the 
idea and applications of WG methods to various problem were provided in \cite{wang2015weak}.

In this paper, in contrast to the projection-based error analysis in \cite{cockburn2010projection,oikawa2015hybridized}, we use the Ladyzhenskaya-Babu\v{s}ka-Brezzi (LBB)  theory to prove two types of uniform stability results under some 
carefully constructed parameter-dependent norms for HDG methods. Based on the uniform stability results, 
we prove uniform and optimal error estimates for HDG methods.  In addition, by using properly 
defined parameter-dependent norms, we further prove two types of uniform stability results for 
WG methods. Similarly based on the uniform stability results, we provide uniform and optimal error estimates for WG methods. These uniform stability results and error estimates for WG methods are meaningful and interesting improvement for 
the results in \cite{wang2013weak,wang2014weak}. Following these uniform stability results for HDG 
methods and WG methods presented in this paper, an HDG method is shown to converge to a 
primal conforming method, whereas a WG method is shown to converge to a mixed conforming method by taking the 
limit of the stabilization parameters.

We illustrate the main idea and results by using the
following elliptic boundary value problem:
\begin{equation} \label{pois}
\left\{
\begin{aligned}
-{\rm div} (\alpha\nabla u)&= f & {\rm in} \ \Omega, \\
 u&=0 & {\rm on} \ \partial\Omega,
\end{aligned}
\right.
\end{equation}
where $\Omega\subset \mathbb R^d$ ($d\ge 1$) is a bounded domain and
$\alpha: \mathbb{R}^d \rightarrow \mathbb{R}^d$ is a bounded and
symmetric positive definite matrix, and its inverse is denoted by $c =
\alpha^{-1}$.  Setting $\vect{p} =- \alpha\nabla u$, the above problem
can be written as:
\begin{equation}\label{H1}
\left\{
\begin{aligned}
c \vect{p} + \nabla u &= 0 & {\rm in}\ \Omega, \\ 
-{\rm div}  \vect{p} &= f & {\rm in}\ \Omega, \\
 u&=0 & {\rm on} \ \partial\Omega.
\end{aligned}
\right.
\end{equation}

The rest of the paper is organized as follows. In Section
\ref{sec:preliminaries}, some preliminary materials are provided. In
Section \ref{sec:framework}, we set up the HDG and WG methods and provide 
the main uniform well-posedness results.  
Based on the uniform well-posedness results, we present uniform and optimal error estimates for HDG and WG in Section
\ref{sec:Error:HDG:WG}, and show that an HDG method converges to a 
primal conforming method, whereas a WG method converges to a mixed conforming method by taking the 
limit of the stabilization parameters in Section \ref{relationship}. In Section \ref{sec:HDG:WG}, we provide proof of the uniform
well-posedness of HDG and WG under the specific parameter-dependent norms.  
We provide a brief summary in the last section.

\section{Preliminaries} \label{sec:preliminaries}
In this section, we describe some basic notation. Throughout this paper, we
use letter $C$ to denote a generic positive constant, which may stand
for different values at different occurrences, but not depending the mesh size and the stability parameters.  
The notations $x
\lesssim y$ and $x \gtrsim y$ mean $x \leq Cy$  and $x \geq Cy$,
respectively. 

\subsection{Discontinuous Galerkin Notation}
Given a bounded domain $D\subset \mathbb{R}^d$ and a positive integer
$m$, $H^m(D)$ is the Sobolev space with the corresponding usual norm
and semi-norm, which are denoted respectively by $\|\cdot\|_{m,D}$ and
$|\cdot|_{m,D}$. We abbreviate them by $\|\cdot\|_{m}$ and
$|\cdot|_{m}$, respectively, when $D$ is chosen as $\Omega$.  The
$L^2$-inner products on $D$ and $\partial D$ are denoted by $(\cdot,
\cdot)_{D}$ and $\langle\cdot, \cdot\rangle_{\partial D}$,
respectively.  Moreover, $\|\cdot\|_{0,D}$ and $\|\cdot\|_{0,\partial D}$ are the
norms of Lebesgue spaces $L^2(D)$ and $L^2(\partial D)$, respectively, and $\|\cdot\|=\|\cdot\|_{0,\Omega}$.
We also set $H({\rm div}, \Omega)=\big\{\boldsymbol u\in \boldsymbol {L^2(\Omega)}:{\rm div} \boldsymbol u\in L^2(\Omega)\big\}$
equipped with the norm $\|\boldsymbol u\|^2_{\rm div}=(\boldsymbol u,\boldsymbol u)+({\rm div}\boldsymbol u,{\rm div}\boldsymbol u)$.

We assume $\Omega$ is a polygonal domain, and a family of triangulations of
$\overline{\Omega}$ is denoted by
$\{\mathcal{T}_h\}_h$, with the minimal angle condition satisfied. Let
$h_K ={\rm diam}(K)$ and $h = \max\{h_K: K\in \mathcal{T}_h\}$.
We denote ${\cal E}_h^i$ the set of interior edges (or faces) of $\mathcal{T}_h$ and ${\cal
E}_h^\partial$ the set of boundary
edges (or faces), and let ${\cal E}_h={\cal
E}_h^\partial\cup {\cal E}_h^i$. For $e\in {\cal E}_h$, let $h_e ={\rm diam}(e)$. For $e\in {\cal E}_h^i$, we choose a fixed normal unit direction denoted by
$\mathbf n_e$, and for $e\in {\cal E}_h^\partial$, we take the outward unit 
normal as $\mathbf n_e$. 
Let $e$ be the common edge of two elements $K^+$ and $K^-$, and
$\vect{n}^i$ = $\vect{n}|_{\partial K^i}$ be the unit outward normal
vector on $\partial K^i$ with $i = +,-$.  For any scalar-valued
function $v$ and vector-valued function $\vect{q}$, let $v^{\pm}$ =
$v|_{\partial K^{\pm}}$ and $\vect{q}^{\pm}$ = $\vect{q}|_{\partial
K^{\pm}}$. Then, we define averages $\{\cdot\}, \{\!\!\{\cdot \}\!\!\}$ 
and jumps $\llbracket \cdot \rrbracket$, $[\cdot]$ as follows:
\begin{align*}
&\{v\} = \frac{1}{2}(v^+ + v^-),\qquad \{\vect{q}\} =
\frac{1}{2}(\vect{q}^+ + \vect{q}^-), \qquad\{\!\!\{\vect q\}\!\!\}=\frac{1}{2}(\vect{q}^+\cdot \vect{n}^+ - \vect{q}^-\cdot\vect{n}^-)\qquad &{\rm on}\ e\in {\cal
  E}_h^i,\\
&\llbracket v \rrbracket  = v^+\vect{n}^+ + v^-\vect{n}^- ,\qquad [v] = v^+ - v^-,
\qquad[\vect{q}] = \vect{q}^+\cdot \vect{n}^+ + \vect{q}^-\cdot \vect{n}^-  \qquad &{\rm on}\ e\in {\cal E}_h^i,\\
&\llbracket v \rrbracket  = v \vect{n}, \qquad[v] = v, \qquad\{\vect{q}\}=\vect{q},\qquad \{\!\!\{\vect q\}\!\!\}=\vect q\cdot \vect n \qquad 
&{\rm on}\ e \in {\cal E}_h^\partial.
\end{align*}
Here,  we specify $\vect{n}$ as the outward unit normal direction on $\partial \Omega$.  

We define some inner products as follows:
\begin{equation}
(\cdot,\cdot)_{\mathcal T_h}=\sum_{K\in \mathcal T_h}(\cdot,\cdot)_{K},~~~~~
\langle\cdot,\cdot\rangle_{\mathcal E_h}=\sum_{e\in \mathcal E_h}\langle\cdot,\cdot\rangle_{e}, ~~~~~\langle\cdot,\cdot\rangle_{\mathcal E^i_h}=\sum_{e\in \mathcal E^i_h}\langle\cdot,\cdot\rangle_{e}, ~~~~\langle\cdot,\cdot\rangle_{\partial\mathcal T_h }=\sum_{K\in \mathcal T_h}\langle\cdot,\cdot\rangle_{\partial K}.
\end{equation}
We now give more details about the last notation of the inner product. For any scalar-valued
function $v$ and vector-valued function $\boldsymbol q$, 
$$
\langle v, \boldsymbol  q \cdot 
\boldsymbol n \rangle_{\partial \mathcal T_h}=\sum_{K\in \mathcal T_h}\langle v, \boldsymbol q\cdot \boldsymbol n\rangle_{\partial K}=\sum_{K\in \mathcal T_h}\langle v, \boldsymbol q\cdot \boldsymbol n_K\rangle_{\partial K}.
$$
Here, we specify the outward unit normal direction $\boldsymbol n$
corresponding to the element $K$, namely $\boldsymbol n_K$. 

For the piecewise smooth scalar-valued function $v$ and vector-valued function
$\vect{q}$, let $\nabla_h$ and ${\rm div}_h$ be defined by the relation
$$
(\nabla_hv)|_K=\nabla (v|_K), \quad
({\rm div}_h\vect{q})|_K={\rm div} (\vect{q}|_K),
$$
on any element $K\in{\cal T}_h$, respectively.

With the definition of averages and jumps, we have the
following identity:
\begin{equation}\label{equ:dg-identity_1}
\langle v, \boldsymbol q\cdot \boldsymbol n\rangle_{\partial T_h}
=\langle \{\boldsymbol q\},\llbracket v\rrbracket\rangle_{\mathcal E_h}+
\langle [\boldsymbol q],\{v\}\rangle_{\mathcal E_h^i},
\end{equation}
and
\begin{equation}\label{equ:dg-identity_2}
\{\vect{q}\}\cdot \llbracket v \rrbracket=\{\!\!\{\vect{q}\}\!\!\}\cdot[v].
\end{equation}




Before discussing various Galerkin methods, we need to introduce the
finite element spaces associated with the triangulation
$\mathcal{T}_h$. First, $V_h$ and $\boldsymbol Q_h$ are the piecewise scalar and 
vector-valued discrete spaces on the triangulation $\mathcal T_h$, respectively and for $k\ge 0$, 
we define the spaces as follows:
\begin{equation}\label{Spaces}
\begin{aligned}
V^{k}_h&=\big\{v_h\in L^2(\Omega): v_h|_{K}\in \mathcal{P}_k(K), \forall
K\in \mathcal T_h \big\},\\
\boldsymbol Q^k_h& = \big\{\boldsymbol p_h\in \boldsymbol {L^2(\Omega)}:
\boldsymbol p_h|_K\in \boldsymbol {\mathcal{P}_k(K)}, \forall K\in
\mathcal T_h \big\},\\
\boldsymbol Q^{k,RT}_h &= \big\{\boldsymbol p_h\in \boldsymbol {L^2(\Omega)}:
\boldsymbol p_h|_K\in \boldsymbol {\mathcal{P}_k(K)}+\boldsymbol x
P_k(K), \forall K\in \mathcal T_h \big\},
\end{aligned}
\end{equation}
where $\mathcal{P}_k(K)$ is the space of polynomial functions of
degree at most $k$ on $K$.  We also use the
following spaces associated with $\mathcal{E}_h$:
\begin{equation}\label{Edge:spaces}
\begin{aligned}
\hat{\boldsymbol{Q}}_h &=\big\{\hat{\boldsymbol p}_h:
\hat{\boldsymbol{p}}_h|_e\in \hat Q(e)\boldsymbol{n}_e, \forall
e\in \mathcal{E}_h\big\},\\
{\hat Q}_h&=\big\{{\hat p}_h: {\hat p}_h|_e\in \hat Q(e), \forall e\in
\mathcal{E}_h\big\}, \\ 
{\hat V}_h &= \big\{{\hat v}_h: {\hat v}_h|_e\in \hat V(e), e\in
\mathcal{E}^i_h, {\hat v}_h|_{\mathcal E_h^{\partial}}=0\big\}, \\
{\hat Q}^k_h&=\big\{{\hat p}_h\in L^2(\mathcal E_h): {\hat
p}_h|_e\in \mathcal{P}_k(e), \forall e\in \mathcal{E}_h\big\}, \\
{\hat V}^k_h&=\big\{{\hat v}_h \in L^2(\mathcal E_h): {\hat
v}_h|_e\in \mathcal{P}_k(e), \forall e\in \mathcal{E}^i_h, {\hat
  v}_h|_{\mathcal E_h^{\partial}}=0\big\},
\end{aligned}
\end{equation}
where $\hat Q(e)$ and $\hat V(e)$ are some local spaces on $e$ and
$\mathcal{P}_k(e)$ is the space of polynomial functions of degree at most $k$ on $e$.
For convenience, we denote $\boldsymbol{\tilde Q}_h = \boldsymbol{Q}_h\times\boldsymbol{\hat
Q}_h$ and $\tilde V_h=V_h\times \hat V_h$.


\section{Uniform Stability for HDG and WG Methods} \label{sec:framework}
In this section, we set up the HDG and WG methods first and then provide the uniform well-posedness for 
both HDG and WG methods under proper parameter-dependent defined norms. 
\subsection{Setting up the HDG and WG Methods}
Now we start with the second-order elliptic equation and
set $\vect {p}=-\alpha\nabla u$ to obtain the following form:
\begin{equation}\label{H11}
\left\{
\begin{aligned}
c \vect{p} + \nabla u &= 0 \qquad  {\rm in}\ \Omega, \\ 
{\rm div} \vect{p} &= f  \qquad {\rm in}\ \Omega. \\
\end{aligned}
\right.
\end{equation}
Multiplying the first and second equations by $\boldsymbol q_h\in
\boldsymbol Q_h$ and $v_h\in V_h$, respectively, then integrating on an
element $K\in \mathcal T_h$, we obtain:
\begin{equation}\label{eq:element}
\left\{
\begin{aligned}
\displaystyle (c\boldsymbol p, \boldsymbol q_h)_K-(u, {\rm div}
\boldsymbol q_h)_K+\langle u, \boldsymbol q_h\cdot \boldsymbol
n_K\rangle_{\partial K} &= 0 \qquad \qquad\quad \forall \boldsymbol
q_h\in \boldsymbol Q_h,\\
\displaystyle (\boldsymbol p, \nabla v_h)_K-\langle \boldsymbol p\cdot
\boldsymbol n_K, v_h\rangle_{\partial K}&=-(f, v_h)_K ~\quad \forall
v_h\in V_h.
\end{aligned}
\right.
\end{equation}
Summing on all $K\in \mathcal T_h$, we have:
\begin{equation}
\left\{
\begin{array}{l}
(c\boldsymbol p, \boldsymbol q_h)_{\mathcal T_h} - (u, {\rm div}_h \boldsymbol q_h)_{\mathcal T_h}
+ \langle u, \boldsymbol q_h\cdot \boldsymbol n\rangle_{\partial
  \mathcal{T}_h} = 0 \qquad \forall \boldsymbol q_h\in \boldsymbol Q_h,\\
(\boldsymbol p, \nabla_h v_h)_{\mathcal T_h} - \langle \boldsymbol p\cdot \boldsymbol
n, v_h\rangle_{\partial \mathcal{T}_h} = -(f, v_h)_{\mathcal T_h} \qquad \forall
v_h\in V_h.
\end{array}
\right.
\end{equation}
Now we approximate $u$, $\boldsymbol p$ by $u_h\in V_h$, and
$\boldsymbol p_h\in \boldsymbol Q_h$, respectively, and the trace of
$u$ and the flux $\boldsymbol p\cdot\boldsymbol n$ on $\partial K$
by $\check u_h$, $ \check{\boldsymbol p}_h \cdot\boldsymbol
n$. Hence, we have:
\begin{equation}\label{eq:common:1}
\left\{
\begin{array}{l}
(c\boldsymbol p_h, \boldsymbol q_h)_{\mathcal T_h} - (u_h, {\rm div}_h \boldsymbol
q_h)_{\mathcal T_h} + \langle \check u_h, \boldsymbol
q_h\cdot \boldsymbol n\rangle_{\partial \mathcal{T}_h} = 0, \forall
\boldsymbol q_h\in \boldsymbol Q_h,\\
(\boldsymbol p_h, \nabla_h v_h)_{\mathcal T_h} - \langle \check{\boldsymbol p}_h\cdot
\boldsymbol n, v_h\rangle_{\partial \mathcal{T}_h} = -(f,
v_h)_{\mathcal T_h} \qquad \forall v_h\in V_h.
\end{array}
\right.
\end{equation}
Next, we need to derive appropriate equations for the variables of
$\check u_h$ and $\boldsymbol {\check p}_h$.  The starting point is
the following relationship:
\begin{equation}\label{hdg-wg}
\boldsymbol {\check p}_h \cdot \boldsymbol n_K + \tau \check u_h
=\boldsymbol p_h \cdot \boldsymbol n_K + \tau u_h, \qquad
\boldsymbol {\check p}_h={\check p}_h\boldsymbol n_e.
\end{equation}
The idea is that we only use either ${\check p}_h$ or
$\check u_h$ as an unknown and then use \eqref{hdg-wg} to determine
the other variable.  There are two different approaches; one approach is for
deriving HDG methods, and the other one is for deriving WG methods.

\paragraph{First approach: (Hybridized Discontinuous Galerkin)}
Set $\check u_h=\hat u_h\in \hat
V_h$ as an unknown that is single-valued.  The ``continuity'' of
$\boldsymbol {\check p}_h$ is then enforced weakly as follows:
\begin{equation}\label{eq:common:4}
\langle\boldsymbol{\check p}_h\cdot\boldsymbol{n},
\hat{v}_h\rangle_{\partial\mathcal T_h}=0,~\forall \hat v_h \in \hat
V_h.
\end{equation}
where $\boldsymbol {\check p}_h$ is given by \eqref{hdg-wg}.  From the
identity \eqref{equ:dg-identity_1} and the fact that $[\hat{v}_h]=0$,
a straightforward calculation shows that \eqref{eq:common:4} can be
rewritten as:
\begin{equation} 
\langle [{\check p}_h], \hat{v}_h\rangle_{\mathcal E_h} := 
\sum_{e\in \mathcal{E}_h} \langle [\check p_h], \hat v_h \rangle_e =0
\qquad \forall \hat v_h\in \hat V_h.
\end{equation}
Collecting \eqref{eq:common:1}, \eqref{hdg-wg}, and \eqref{eq:common:4}, the HDG methods read:
Find $(\boldsymbol{p}_h,  \tilde{u}_h)\in \boldsymbol Q_h\times \tilde{V}_h$ such that for any
 $(\boldsymbol{q}_h, \tilde{v}_h)\in \boldsymbol Q_h\times \tilde{V}_h$,
 \begin{equation}\label{Stabilizedpweak}
\left\{
\begin{array}{l}
a_h(\boldsymbol{p}_h,\boldsymbol{q}_h)+b_h(\boldsymbol{ q}_h,\tilde {u}_h)=0,\\
b_h(\boldsymbol{p}_h, \tilde v_h)+c_h( \tilde{u}_h, \tilde{v}_h)=-(f,v_h)_{\mathcal T_h}.
\end{array}
\right.
\end{equation}
Here
 \begin{equation}\label{definition:ahbhch}
\left\{
\begin{array}{l}
a_h(\boldsymbol{p}_h,\boldsymbol{q}_h)=(c \boldsymbol{p}_h, \boldsymbol{q}_h)_{\mathcal T_h},\\
b_h(\boldsymbol{q}_h, \tilde u_h)=-(u_h, {\rm div} \boldsymbol{q}_h)_{\mathcal T_h}+\langle\hat{u}_h, \boldsymbol{q}_h\cdot\boldsymbol{n}_K\rangle_{\partial {\mathcal T_h}},\\
c_h( \tilde{u}_h, \tilde{v}_h)=-\tau\langle u_h-\hat{u}_h, v_h-\hat{v}_h\rangle_{\partial {\mathcal T_h}},
\end{array}
\right.
\end{equation}
where $\tau$ is the stabilization parameter. 

The HDG method can be written in a compact form:  Find $(\boldsymbol{p}_h,  \tilde{u}_h)\in \boldsymbol Q_h\times \tilde{V}_h$ such that for any
 $(\boldsymbol{q}_h, \tilde{v}_h)\in \boldsymbol Q_h\times \tilde{V}_h$,
\begin{equation}
A_h((\boldsymbol{p}_h, \tilde u_h),(\boldsymbol{q}_h, \tilde v_h))=-(f,v_h)_{\mathcal T_h},
\end{equation}
where 
\begin{equation}\label{HDG:Ah}
A_h((\boldsymbol{p}_h, \tilde u_h),(\boldsymbol{q}_h, \tilde v_h))=a_h(\boldsymbol{p}_h,\boldsymbol{q}_h)+b_h(\boldsymbol{ q}_h,\tilde {u}_h)+b_h(\boldsymbol{p}_h, \tilde v_h)+c_h( \tilde{u}_h, \tilde{v}_h).
\end{equation}

In the first case, we choose $\tau=\rho h_K$ in \eqref{definition:ahbhch} and for 
any $\tilde v\in \tilde V_h$ and $\boldsymbol{q}_h\in \boldsymbol{Q}_h$, we define
\begin{equation}\label{HDG:div:norm}
\left\{
\begin{array}{l}
\|\tilde v_h\|_{0, \rho,h}^2=(v_h,v_h)_{\mathcal T_h}+\rho \sum\limits_{e\in \mathcal{E}^i_h}h_e\langle\hat v_h,\hat v_h\rangle_e,\\
\|\boldsymbol{q}_h\|_{{\rm div},\rho,h}^2=(c \boldsymbol{q}_h, \boldsymbol{q}_h)_{\mathcal T_h}+( {\rm div}\boldsymbol{q}_h, {\rm div} \boldsymbol{q}_h)_{\mathcal T_h}+\rho^{-1}\sum\limits_{e\in \mathcal{E}^i_h}h_e^{-1} \langle\hat P_e([\boldsymbol{q}_h]),\hat P_e([\boldsymbol{q}_h])\rangle_e,
\end{array}
\right.
\end{equation}
where $\hat P_e: L^2(e)\rightarrow \hat V(e)$ is the $L^2$ projection.

In the second case, we choose $\tau=\rho^{-1}h_K^{-1}$ in \eqref{definition:ahbhch} and for 
any $\tilde v\in \tilde V_h$ and $\boldsymbol q_h\in \boldsymbol Q_h $, we define 
\begin{equation} \label{HDG:grad:norm}
\|\tilde v_h\|^2_{\tilde 1,\rho,h}=(\nabla_h v_h,\nabla_h v_h)_{\mathcal T_h}+\rho^{-1} \sum\limits_{K\in \mathcal{T}_h}h_K^{-1}\langle v_h-\hat{v}_h, v_h-\hat{v}_h\rangle_{\partial K},
~~~~\|\boldsymbol q_h\|^2=(c\boldsymbol q_h,\boldsymbol q_h)_{\mathcal T_h}.
\end{equation}
By noting that
$$
\langle v_h-\hat{v}_h, v_h-\hat{v}_h\rangle_{\partial {\mathcal T_h}}=2\langle\{v_h-\hat{v}_h\}, \{v_h-\hat{v}_h\}\rangle_{\mathcal{E}_h}+
\frac{1}{2}\langle \lbrack\!\lbrack v_h- \hat{v}_h\rbrack\!\rbrack, \lbrack\!\lbrack v_h-\hat{v}_h\rbrack\!\rbrack\rangle_{\mathcal{E}_h},
$$ 
$\|\tilde v_h\|_{\tilde 1,\rho,h}$ is indeed a norm on $\tilde V_h$.

\paragraph{Second approach: (Weak Galerkin)}
We set $\boldsymbol{\check p}_h := \boldsymbol{\hat p}_h = \hat
p_h\boldsymbol n_e\in \boldsymbol {\hat Q}_h $ as an unknown that is
single-valued.  The ``continuity'' of $\check u_h$ is
then enforced weakly as follows:
\begin{equation}\label{eq:common:2}
\langle{\check u}_h,\boldsymbol{\hat q}_h\cdot \boldsymbol
n\rangle_{\partial\mathcal T_h}=0 \qquad \forall \boldsymbol{\hat q}_h
\in \boldsymbol{\hat Q}_h,
\end{equation}
where $\check u_h$ is again given by \eqref{hdg-wg}.  From the
identity \eqref{equ:dg-identity_1} and the fact that
$[\boldsymbol{\hat q}_h]=0$, a straightforward calculation shows that
\eqref{eq:common:2} can be rewritten as:
\begin{equation}\label{jump:mean:product}
\langle[{\check u_h}], {\hat q}_h\rangle_{\mathcal E_h} := \sum_{e\in
\mathcal{E}_h}\langle[{\check u_h}], {\hat q}_h\rangle_{e}=0 \qquad
\forall {\hat q}_h\in {\hat Q}_h.
\end{equation}
Collecting \eqref{eq:common:1}, \eqref{hdg-wg}, and \eqref{eq:common:2}, the WG methods read: Find $(\boldsymbol{\tilde p}_h, u_h)\in \boldsymbol{\tilde Q}_h\times V_h$ such that for any
 $(\boldsymbol{\tilde q}_h, v_h)\in \boldsymbol{\tilde Q}_h\times V_h$,
\begin{equation}\label{Stabilizedpweak_WG}
\left\{
\begin{array}{l}
a_w(\boldsymbol{\tilde p}_h,\boldsymbol{\tilde q}_h)+b_w(u_h,\boldsymbol{\tilde q}_h)=0,\\
b_w(\boldsymbol{\tilde p}_h,v_h)=-(f,v_h).
\end{array}
\right.
\end{equation}
Here
 \begin{equation}\label{WG_aform}
\left\{
\begin{array}{l}
a_w(\boldsymbol{\tilde p}_h,\boldsymbol{\tilde q}_h)=(c \boldsymbol{p}_h, \boldsymbol{q}_h)_{\mathcal T_h}+\eta \langle(\boldsymbol{p}_h-\boldsymbol{\hat p}_h)\cdot \boldsymbol{n}, (\boldsymbol{q}_h-\boldsymbol{\hat q}_h)\cdot \boldsymbol{n}\rangle_{\partial {\mathcal T_h}},\\
b_w(\boldsymbol{\tilde p}_h,v_h)=(\boldsymbol{p}_h, \nabla_h v_h)_{\mathcal T_h}-( \boldsymbol{\hat p}_h\cdot\boldsymbol{n}_K,{v}_h)_{\partial \mathcal T_h},
\end{array}
\right.
\end{equation}
where $\eta$ is the stabilized parameter.


The WG method can be rewritten in a compact form: Find $(\boldsymbol{\tilde p}_h, u_h)\in \boldsymbol{\tilde Q}_h\times V_h$ such that for any
 $(\boldsymbol{\tilde q}_h, v_h)\in \boldsymbol{\tilde Q}_h\times V_h$:
\begin{equation}
A_w((\boldsymbol{ \tilde p}_h, u_h),(\boldsymbol{ \tilde q}_h, v_h))=-(f,v_h),
\end{equation}
where 
\begin{equation}\label{WG:Aw}
A_w((\boldsymbol{ \tilde p}_h, u_h),(\boldsymbol{ \tilde q}_h, v_h))=a_w(\boldsymbol{\tilde p}_h,\boldsymbol{\tilde q}_h)+b_w(\boldsymbol{\tilde q}_h,u_h)+b_w(\boldsymbol{\tilde p}_h,v_h).
\end{equation}

In the first case, we choose parameter $\eta$ as $\eta =\rho h_K$ in \eqref{WG_aform} and 
for any $v_h\in V_h$ and $\boldsymbol{\tilde q}_h\in \boldsymbol{\tilde Q}_h$, 
we define the norms as follows:
\begin{equation}\label{WG:grad:norm}
\left\{
\begin{array}{l}
\|v_h\|^2_{1,h,\rho}=\|\nabla_h v_h\|^2+\rho^{-1}\sum\limits_{e\in \mathcal{E}_h}h_e^{-1}\|\hat Q_e([v_h])\|^2_{0,e},\\
\|\boldsymbol{\tilde q}_h\|^2_{0,h,\rho}=(c \boldsymbol{q}_h, \boldsymbol{q}_h)_{\mathcal T_h}+\rho\sum\limits_{K\in \mathcal{T}_h}h_K\langle(\boldsymbol{q}_h-\boldsymbol{\hat q}_h)\cdot \boldsymbol{n}_K, (\boldsymbol{q}_h-\boldsymbol{\hat q}_h)\cdot \boldsymbol{n}_K\rangle_{\partial K}.
\end{array}
\right.
\end{equation}
where $\hat Q_e$ is the $L^2$ projection from $L^2(e)$ to $\hat Q(e)$.

In the second case, we choose parameter $\eta$ as $\eta =\rho^{-1} h^{-1}_K$ in \eqref{WG_aform} 
and for any $v_h\in V_h$ and $\boldsymbol{\tilde q}_h\in \boldsymbol{\tilde Q}_h$, we define the norms 
as follows:
\begin{equation}\label{WG:div:norm}
\|u_h\|^2=(u_h,u_h)_{\mathcal T_h}; ~~ \|\boldsymbol{\tilde q}_h\|^2_{\widetilde {\rm div},\rho,h}=
\sum_{K\in \mathcal{T}_h}\|\boldsymbol{\tilde q}_h\|^2_{\widetilde{\rm div},\rho,K},
\end{equation}
where 
$\|\boldsymbol{\tilde q}_h\|^2_{{\widetilde {\rm div}},\rho,K}
=(c \boldsymbol{q}_h,
\boldsymbol{q}_h)_K+
({\rm div}\boldsymbol{q}_h, {\rm div}\boldsymbol{q}_h)_K+\rho^{-1}
h^{-1}_K\langle(\boldsymbol{q}_h-\boldsymbol{\hat q}_h)\cdot \boldsymbol{n}_K, (\boldsymbol{q}_h-\boldsymbol{\hat q}_h)\cdot \boldsymbol{n}_K\rangle_{\partial K}$.

\subsection{Uniform Well-posedness of HDG and WG} 
For the elliptic problem \eqref{H1}, we set a discretization: 
Find $\boldsymbol{\mathcal U_h}\in \boldsymbol{U_h}$, such that:
\begin{equation}\label{general:formulation}
A_{h,\theta}(\boldsymbol{\mathcal U_h},\boldsymbol{\mathcal V_h})=F(\boldsymbol{\mathcal V_h})~~~\forall~\boldsymbol{\mathcal V_h}\in \boldsymbol{U_h},
\end{equation} 
where $\boldsymbol{U_h}$ is a finite dimensional space according to partition $\mathcal T_h$ and $A_{h,\theta}(\boldsymbol{\mathcal U_h},\boldsymbol{\mathcal V_h})$ is a general symmetric $\theta$-parameter-dependent bilinear form and $F(\boldsymbol{\mathcal V_h})=-(f, v_h)_{\mathcal T_h}$.

Let $\boldsymbol{\mathcal U}=(\boldsymbol {p}, u)$ be the true solution of  \eqref{H1}. 
\begin{enumerate}
\item We say that the discretization \eqref{general:formulation} is consistent if 
\begin{equation}\label{general: consistence}
A_{h,\theta}(\boldsymbol{\mathcal U},\boldsymbol{\mathcal V_h})=F(\boldsymbol{\mathcal V_h})~~~\forall~\boldsymbol{\mathcal V_h}\in \boldsymbol{U_h}.
\end{equation}
\item We say that the bilinear form $A_{h,\theta}(\boldsymbol{\mathcal U_h},\boldsymbol{\mathcal V_h})$ is uniformly continuous with respect to the norm $\|\cdot\|_{\boldsymbol U_{h,\theta}}$ if 
\begin{equation}\label{uniform:continuous:Au}
|A_{h,\theta}(\boldsymbol{\mathcal U_h},\boldsymbol{\mathcal V_h})|\leq M_0 \|\boldsymbol{\mathcal U_h}\|_{\boldsymbol U_{h,\theta}}\|\boldsymbol{\mathcal V_h}\|_{\boldsymbol U_{h,\theta}},
\end{equation}
where $M_0$ is independent of the parameter $\theta$ and the mesh size $h$.
\item We say that the bilinear form $A_{h,\theta}(\boldsymbol{\mathcal U_h},\boldsymbol{\mathcal V_h})$ 
satisfies the inf-sup condition uniformly with respect to the norm $\|\cdot\|_{\boldsymbol U_{h,\theta}}$ 
if there exists a constant $\beta_1>0$ that does not depend on the parameter $\theta$ and the mesh size 
$h$ such that:
\begin{equation} \label{uniform:inf-sup:Au} 
\inf_{\boldsymbol{\mathcal V_h}\in \boldsymbol U_h} \sup_{\boldsymbol{\mathcal U_h}\in \boldsymbol U_h}   \frac{A_{h,\theta}(\boldsymbol{\mathcal U_h},\boldsymbol{\mathcal V_h})}{\|\boldsymbol{\mathcal U_h}\|_{\boldsymbol U_{h,\theta}}\|\boldsymbol{\mathcal V_h}\|_{\boldsymbol U_{h,\theta}}}   \geq \beta_1.
\end{equation}
\end{enumerate}
By the C\'ea's lemma, see \cite{Ciarlet1978}, we have 
\begin{theorem} \label{uniform:error:Au:th}
If a discretization \eqref{general:formulation} satisfies 
\begin{enumerate}
\item consistency, namely \eqref{general: consistence};
\item continuity uniformly, namely \eqref{uniform:continuous:Au};
\item  inf-sup condition uniformly with respect to the norm $\|\cdot\|_{\boldsymbol U_{h,\theta}}$, namely \eqref{uniform:inf-sup:Au},
\end{enumerate}
then we have 
\begin{equation} \label{uniform:error:Au}
\|\boldsymbol{\mathcal U}-\boldsymbol{\mathcal U_h}\|_{\boldsymbol U_{h,\theta}}\leq C_1\inf_{\boldsymbol{\mathcal V_h}\in \boldsymbol U_h}\|\boldsymbol{\mathcal U}-\boldsymbol{\mathcal V_h}\|_{\boldsymbol U_{h,\theta}},
\end{equation}
where $C_1$ is independent of the parameter $\theta$ and the mesh size $h$. Further, we say the discretization \eqref{general:formulation} is uniformly stable. 
\end{theorem}

Now for the HDG method, the parameter $\theta=\tau$ in \eqref{general:formulation} and the bilinear form is given by \eqref{HDG:Ah}, the space 
$\boldsymbol U_h=\boldsymbol Q_h\times {\tilde V_h},~ \boldsymbol{\mathcal U_h}=(\boldsymbol p_h, \tilde u_h)$.
In the first case, the parameter $\tau=\rho h_K$, and the norm 
$\|\boldsymbol{\mathcal U_{h}}\|^2_{\boldsymbol U_{h,\theta}}=\|\boldsymbol{p}_h\|_{{\rm div},\rho,h}^2+\|\tilde u_h\|_{0, 
\rho,h}^2$. In the second case, the parameter $\tau=\rho^{-1} h^{-1}_K$ and the norm $\|\boldsymbol{\mathcal U_{h}}\|^2_{\boldsymbol U_{h,\theta}}=
\|\boldsymbol p_h\|^2+\|\tilde u_h\|^2_{\tilde 1,\rho,h}$.
\begin{theorem}\label{wellposed:Ah:divgrad}
 We have two uniform stability results for the HDG method as follows:
\begin{enumerate}
\item \label{wellposed:Ah:div} For any $0<\rho\leq 1$,  and for $k\ge 0$, 
if $\boldsymbol Q_h=\boldsymbol Q_h^{k+1},
V_h=V_h^k$ and $\hat V_h=\hat V_h^r$ where $0\le r\le k+1$, or $\boldsymbol
Q_h=\boldsymbol Q_h^{k,RT}, V_h=V_h^k$ and $\hat V_h=\hat V_h^r$ where
$0\le r\le k$, then the bilinear form $A_h((\cdot, \cdot),(\cdot, \cdot))$ with $\tau=\rho h_K$ is uniformly stable with respect to the norms defined by \eqref{HDG:div:norm};
\item \label{wellposed:Ah:grad}  Assume that $\nabla_h V_h \subset \boldsymbol Q_h$, then there exists a positive constant $\rho_0$ such that for any $0<\rho\leq \rho_0$ the bilinear form $A_h((\cdot, \cdot),(\cdot, \cdot))$ with $\tau=\rho^{-1} h_K^{-1}$ is uniformly stable with respect to the norms defined by \eqref{HDG:grad:norm}. 
\end{enumerate} 
\end{theorem}

From part \ref{wellposed:Ah:grad} of the above theorem, we have the following corollary:
\begin{corollary}\label{uniform_stable}
Assume $\nabla_h V_h \subset \boldsymbol Q_h$, then there exists a unique solution $(\boldsymbol{p}_h,  \tilde{u}_h)\in \boldsymbol Q_h\times \tilde{V}_h$ that satisfies \eqref{Stabilizedpweak} with $\tau=\rho^{-1}h_K^{-1}$, and there exists a positive 
constant $\rho_0$ such that for any $0<\rho\leq \rho_0$ the following estimate holds:
\begin{equation}
 \|\boldsymbol{p}_h\|+\|\tilde u_h\|_{\tilde 1,\rho,h}\leq C_2 \|f\|_{*,\rho},
\end{equation}
where $C_2$ is a constant independent of $\rho$ and $h$ and $ \|f\|_{*,\rho}=\sup\limits_{\tilde v_h \in \tilde {V}_h}\frac{(f, v_h)_{\mathcal T_h}}{\|\tilde v_h\|_{\tilde 1,\rho,h}}$.
\end{corollary}
\begin{remark}\label{uniform_stability}
From the above corollary and the discrete Poincar{\'e}--Friedrichs inequalities for piecewise $H^1$ functions \cite{brenner2003poincare},
that is $\|v_h\|\lesssim \|\nabla_h v_h\|+\sum\limits_{e\in \mathcal{E}_h}h^{-1}_e\|\lbrack\!\lbrack v_h\rbrack\!\rbrack\|_{0,e}$, we further have 
$\|\boldsymbol{p}_h\|+\|\tilde u_h\|_{\tilde 1,\rho,h}\leq C_2 \|f\|$.
\end{remark}
\begin{remark}
By the uniform stability results of the HDG method, namely Corollary \ref{uniform_stable}, we 
can prove that the solution of the HDG method converges to the solution of the primal conforming method 
when the parameter $\rho$ approaches to zero, see Section \ref{relationship}.  
\end{remark}
Next, for the WG method, the parameter $\theta=\eta$ in \eqref{general:formulation} and the bilinear 
form is given by \eqref{WG:Aw}, and the space $\boldsymbol U_h=\boldsymbol {\tilde Q_h}\times {V_h},~ \boldsymbol{\mathcal U_h}=(\boldsymbol {\tilde p_h}, u_h)$.
In first case, the parameter $\eta=\rho h_K$ and the norm 
$\|\boldsymbol{\mathcal U_{h}}\|^2_{\boldsymbol U_{h,\theta}}=\|\boldsymbol{\tilde p}_h\|^2_{0,h,\rho}+\|u_h\|^2_{1,h,\rho}$. In the second case, the parameter $\eta=\rho^{-1} h^{-1}_K$, and the norm $\|\boldsymbol{\mathcal U_{h}}\|^2_{\boldsymbol U_{h,\theta}}=\|\boldsymbol{\tilde p}_h\|^2_{\widetilde{\rm div},\rho,h}+\|u_h\|^2$.
\begin{theorem}\label{wellposed:Aw:graddiv} 
We have two uniform stability results for the WG method as follows:
\begin{enumerate}
\item  \label{wellposed:Aw:grad} Assume $\nabla_h V_h\subset \boldsymbol Q_h$,  then for any $0<\rho\leq 1$ the bilinear form $A_w((\cdot, \cdot),(\cdot, \cdot))$ with $\eta=\rho h_K$ is uniformly stable with respect to the norms defined by \eqref{WG:grad:norm};
\item  \label{wellposed:Aw:div} Let $\boldsymbol R_h\subset H({\rm div}, \Omega)\cap
\boldsymbol Q_h$ be the Raviart-Thomas finite element space. Assume
that $\{\!\!\{\boldsymbol R_h\}\!\!\}\subset \hat Q_h$ and
$V_h = {\rm div}_h \boldsymbol Q_h$, then for any $0<\rho\leq 1$ the bilinear form $A_w((\cdot, \cdot),(\cdot, \cdot))$ with $\eta=\rho^{-1} h^{-1}_K$  is uniformly stable with respect to the norms defined by \eqref{WG:div:norm}. 
\end{enumerate}
\end{theorem}
From part \ref{wellposed:Aw:grad} of the above theorem, we have the following corollary:
\begin{corollary}\label{uniform_Stable_WG_Hong}
Assume $\nabla_h V_h\subset \boldsymbol Q_h$, then there exists a unique solution 
$(\boldsymbol{ \tilde p}_h,  {u}_h)\in \boldsymbol  {\tilde Q}_h\times {V}_h$ that 
satisfies \eqref{Stabilizedpweak_WG} with $\eta=\rho h_K$, and for any $0<\rho\leq 1$ 
the following estimates holds:
\begin{equation}
 \|\boldsymbol{\tilde p}_h\|_{0,h,\rho}+\|u_h\|_{1,h,\rho}\leq C_3 \|f\|_{*,\rho},
\end{equation}
where $C_3$ is a constant uniform with respect to $\rho$ and $h$ and $ \|f\|_{*,\rho}=\sup\limits_{v_h \in \tilde {V}_h}\frac{(f, v_h)_{\mathcal T_h}}{\|v_h\|_{1,h,\rho}}$.
\end{corollary}
\begin{remark}
From the above theorem, we improved the result in \cite{wang2014weak} by proving the well-posedness 
of the WG method for any $0<\rho \leq1$, while in \cite{wang2014weak} the inf-sup condition for some constant $\rho$ (for example $\rho=1$) was proved. 
\end{remark}
From part \ref{wellposed:Aw:div} of Theorem \ref{wellposed:Aw:graddiv}, we have the following corollary: 
\begin{corollary}\label{uniform_stable_WG}
Assume the spaces $ \boldsymbol {\tilde Q}_h\times {V}_h$ satisfy the conditions in part \ref{wellposed:Aw:div} of Theorem \ref{wellposed:Aw:graddiv}, then there exists a unique solution $(\boldsymbol{\tilde p}_h,  {u}_h)\in \boldsymbol {\tilde Q}_h\times {V}_h$ that 
satisfies \eqref{Stabilizedpweak_WG} with $\eta=\rho^{-1}h_K^{-1}$, and for any $0<\rho\leq 1$ 
the following estimates holds:
\begin{equation}
 \|\boldsymbol{\tilde p}_h\|_{\widetilde{\rm div},h,\rho}+\|u_h\|\leq C_4 \|f\|,
\end{equation}
where $C_4$ is a uniform constant with respect to $\rho$ and $h$.
\end{corollary}
\begin{remark}
By the above uniform stability result of the WG method, namely Corollary \ref{uniform_stable_WG}, 
we can prove that the solution of the WG method converges to the solution of the mixed conforming method when 
the parameter $\rho$ approaches to zero, see Section \ref{relationship}.  
\end{remark}

\section{Uniform Error Estimates of HDG and WG}\label{sec:Error:HDG:WG}
In this section, based on the uniform stability results shown in Section \ref{sec:framework}, 
we provide the error analysis for HDG and WG methods and obtain uniformly optimal error estimates for 
HDG and WG methods.

\subsection{Error Estimate of HDG Method}
\begin{theorem}\label{Error:uniform_HDG:mixed:rate}
Let $(\boldsymbol p,  u)\in H({\rm div}, \Omega)\times L^2(\Omega)$ be the solution of \eqref{H1} and $\boldsymbol p\in \boldsymbol {H^{k+1}(\Omega)},{\rm div} \boldsymbol p\in H^{k+1}(\Omega), u\in H^{k+1}(\Omega) (k\ge 0)$, and $(\boldsymbol{p}_h,  \tilde{u}_h)\in \boldsymbol {Q}_h\times \tilde{V}_h$ be the solution of \eqref{Stabilizedpweak} with $\tau=\rho h_K $. If we choose the spaces $V_h\times\boldsymbol Q_h\times \hat V_h=V_h^k\times\boldsymbol Q_h^{k,RT}\times \hat V_h^k$, then for any $0<\rho\le1$ the following estimate holds:
\begin{equation}
 \|\boldsymbol p-\boldsymbol{ p}_h\|_{{\rm div},\rho,h}+\|u-\tilde {u}_h\|_{0,\rho,h}\leq C_{r,1} h^{k+1} (|\boldsymbol p|_{k+1}+|{\rm div}\boldsymbol p|_{k+1}+|u|_{k+1}),
\end{equation}
where $ C_{r,1}$ is a constant independent of $h$ and $\rho$.
\end{theorem}
\begin{proof}
From part \ref{wellposed:Ah:div} of Theorem \ref{wellposed:Ah:divgrad} and Theorem \ref{uniform:error:Au:th}, we have 
\begin{equation}
 \|\boldsymbol p-\boldsymbol{ p}_h\|_{{\rm div},\rho,h}+\|u-\tilde {u}_h\|_{0,\rho,h}\lesssim  \inf\limits_{\boldsymbol{q}_h\in \boldsymbol{Q}_h, \tilde v_h\in \tilde V_h } \Big(\|\boldsymbol p-\boldsymbol{ q}_h\|_{{{\rm div}, \rho,h}}+\|u-\tilde v_h\|_{0,\rho,h}\Big).
 \end{equation}
Hence we need to estimate:
\begin{equation}
 \inf\limits_{\boldsymbol{q}_h\in \boldsymbol{Q}_h, \tilde v_h\in \tilde V_h } \Big(\|\boldsymbol p-\boldsymbol{ q}_h\|_{{{\rm div}, \rho,h}}+\|u-\tilde v_h\|_{0,\rho,h}\Big).
\end{equation}
Now we choose $\boldsymbol q_h=\pi^{\rm div}_h\boldsymbol p\in \boldsymbol Q_h$, 
where $\pi^{\rm div}_h$ be the interpolation of $\boldsymbol p$ into the $H({\rm div})$-conforming  Raviart-Thomas ($RT$) finite element space, namely $\boldsymbol q_h\in \boldsymbol Q_h\cap H({\rm div}, \Omega)$. Since $\boldsymbol q_h\cdot \boldsymbol n$ is single-valued,  then by the approximation property of the $RT$ finite element space, we obtain:
\begin{equation}\label{Error:uniform_HDG:mixed:p}
\begin{split}
\|\boldsymbol p-\boldsymbol{q}_h\|^2_{{\rm div}, \rho,h}
= &(c ( \boldsymbol{p}-  \boldsymbol{q}_h), \boldsymbol{p}- \boldsymbol{q}_h)_{\mathcal T_h}+( {\rm div} (\boldsymbol{p}-\boldsymbol{q}_h), {\rm div}( \boldsymbol{p}- \boldsymbol{q}_h))_{\mathcal T_h}\\
&+\rho^{-1}\sum\limits_{e\in \mathcal{E}^i_h}h_e^{-1} \langle\hat P_e([ \boldsymbol{p}-\boldsymbol{q}_h]),\hat P_e([ \boldsymbol{p}-\boldsymbol{q}_h])\rangle_e\\
= &(c ( \boldsymbol{p}-  \boldsymbol{q}_h), \boldsymbol{p}- \boldsymbol{q}_h)_{\mathcal T_h}+( {\rm div} (\boldsymbol{p}-\boldsymbol{q}_h), {\rm div}( \boldsymbol{p}- \boldsymbol{q}_h))_{\mathcal T_h}\\
\lesssim & h^{2k+2}(|\boldsymbol p |^2_{k+1}+|{\rm div}\boldsymbol p|^2_{k+1}).
\end{split}
\end{equation}
Further, we choose $v_h=Q_h (u), \hat v_h=\{Q_h(u)\}$, where $Q_h$ is $L^2$ projection from $L^2(\Omega)$ to $V_h$. 
Then, by using the approximation of $L^2$ projection, trace 
inequality and noting that $0<\rho\le 1$, we have:
\begin{equation}\label{Error:uniform_HDG:mixed:u}
\begin{split}
\|u-\tilde{v}_h\|^2_{0, \rho,h}&=(u-v_h, u-v_h)_{\mathcal T_h}+\rho \sum\limits_{e\in \mathcal{E}^i_h}h_e\langle u-\hat v_h,u-\hat v_h\rangle_e\\
&=(u-Q_h(u), u-Q_h(u))_{\mathcal T_h}+\rho \sum\limits_{e\in \mathcal{E}^i_h}h_e\langle u-\{Q_h(u)\},u-\{Q_h(u)\}\rangle_e\\
&\lesssim \|u-Q_h(u)\|^2+\rho \sum\limits_{e\in \mathcal{E}^i_h}h_e\big(h_e^{-1}\|u-Q_h(u)\|_{K_{e,1}\cup K_{e,2}}^2+h_e\|\nabla_h(u-Q_h(u))\|_{K_{e,1}\cup K_{e,2}}^2\big)\\
&\lesssim \|u-Q_h(u)\|^2+h^2_e\|\nabla_h(u-Q_h(u))\|^2\lesssim h^{2k+2}|u|^2_{k+1},
\end{split}
\end{equation}
where $K_{e,1},K_{e,2}$ are the elements sharing the edge $e$.
 
Combining \eqref{Error:uniform_HDG:mixed:p} and \eqref{Error:uniform_HDG:mixed:u}, we get the desired result.
\end{proof}

\begin{theorem}\label{Error:uniform_HDG:primal:rate}
Let $(\boldsymbol p,  u)\in \boldsymbol {L^2(\Omega)}\times H^1(\Omega) $ be the solution of \eqref{H1} and $\boldsymbol p\in \boldsymbol {H^{k+1}(\Omega)}, u\in H^{k+2}(\Omega)$ $ (k\ge 0)$, and $(\boldsymbol{p}_h,  \tilde{u}_h)\in \boldsymbol {Q}_h\times \tilde{V}_h$ be the solution of \eqref{Stabilizedpweak} with $\tau=\rho^{-1}h_K^{-1} $. If we choose the spaces $V_h\times\boldsymbol Q_h\times \hat V_h= V_h^{k+1}\times\boldsymbol Q_h^{k}\times \hat V_h^{k+1}$, then 
there exists $\rho_0>0$ such that for any $0<\rho\le \rho_0$ the following estimate holds:
\begin{equation}
 \|\boldsymbol p-\boldsymbol{ p}_h\|+\|u-\tilde {u}_h\|_{\tilde 1,\rho,h}\leq C_{r,2} h^{k+1} (|\boldsymbol p|_{k+1}+|u|_{k+2}),
\end{equation}
where $ C_{r,2}$ is independent of $h$ and $\rho$.
\end{theorem}
\begin{proof}
From part \ref{wellposed:Ah:grad} of Theorem \ref{wellposed:Ah:divgrad} and Theorem \ref{uniform:error:Au:th},  we have 
\begin{equation}
 \|\boldsymbol p-\boldsymbol{p}_h\|+\|u-\tilde u_h\|_{\tilde 1,\rho,h}\lesssim \inf\limits_{\boldsymbol{q}_h\in \boldsymbol{Q}_h, \tilde v_h\in \tilde V_h } \Big(\|\boldsymbol p-\boldsymbol{ q}_h\|+\|u-\tilde v_h\|_{\tilde 1,\rho,h}\Big).
\end{equation}
Hence we need to estimate:
\begin{equation}
 \inf\limits_{\boldsymbol{q}_h\in \boldsymbol{Q}_h, \tilde v_h\in \tilde V_h } \Big(\|\boldsymbol p-\boldsymbol{ q}_h\|+\|u-\tilde v_h\|_{\tilde 1,\rho,h}\Big).
\end{equation}
Now we choose $\boldsymbol q_h=Q_h(\boldsymbol p)$, where $Q_h$ is $L^2$ projection from $\boldsymbol {L^2(\Omega)}$ to $\boldsymbol Q_h$. Then, by using the approximation of $L^2$ projection, we obtain:
\begin{equation}\label{Error:uniform_HDG:primal:p}
\|\boldsymbol p-\boldsymbol{q}_h\|\lesssim h^{k+1}|\boldsymbol p|_{k+1}.
\end{equation}
Further, we choose $v_h=\pi_h u$, where $\pi_h$ is the interpolation of $u$ to the continuous finite element 
space, namely $v_h\in V_h\cap H^1_0(\Omega)$. Since $v_h$ is in $H^1_0(\Omega)$, we can choose $\hat v_h$ such that $\hat v_h|_{\partial K}=v_h|_{\partial K}$, for any $K\in \mathcal T_h$, then we get the following convergence rate result:
\begin{equation}\label{Error:uniform_HDG:primal:u}
\begin{split}
\|u-\tilde {v}_h\|^2_{\tilde 1,\rho,h}
=&(\nabla_h(u- v_h),\nabla_h(u-v_h))_{\mathcal T_h}\\
&+\rho^{-1} \sum\limits_{K\in \mathcal{T}_h}h_K^{-1}\langle (u-v_h)-(u-\hat{v}_h), (u-v_h)-(u-\hat{v}_h)\rangle_{\partial K}\\
=&(\nabla(u- \pi_h u),\nabla(u-\pi_h u))_{\mathcal T_h}\lesssim h^{2k+2}|u|^2_{k+2}.
\end{split}
\end{equation}
Combining \eqref{Error:uniform_HDG:primal:p} and \eqref{Error:uniform_HDG:primal:u}, we get the desired result.
\end{proof}

\subsection{Error Estimate of the WG Method}
\begin{theorem}\label{WG-order:grad}
Let $(\boldsymbol p,  u)\in \boldsymbol {L^2(\Omega)}\times H^1(\Omega) $ be the solution of \eqref{H1} and $\boldsymbol p\in \boldsymbol {H^{k+1}(\Omega)}, u\in H^{k+2}(\Omega) $, and $(\boldsymbol{\tilde p}_h,  {u}_h)\in \boldsymbol {\tilde Q}_h\times {V}_h$ be the solution of \eqref{Stabilizedpweak_WG} with $\eta=\rho h_K$. If we choose the spaces ${V}_h\times\boldsymbol {Q}_h\times\hat Q_h={V}_h^{k+1}\times\boldsymbol {Q}_h^k\times\hat Q_h^k$, then for any $0<\rho\leq 1$ the following estimate holds:
\begin{equation}
 \|\boldsymbol p-\boldsymbol {\tilde p}_h\|_{0,h,\rho}+\|u-{u}_h\|_{1,h,\rho}\leq C_{r,3} h^{k+1} (|\boldsymbol p|_{k+1}+|u|_{k+2}),
\end{equation}
where $ C_{r,3}$ is independent of $h$ and $\rho$.
\end{theorem}
\begin{proof}
From part $1$ of Theorem \ref{wellposed:Aw:graddiv} and Theorem \ref{uniform:error:Au:th}, we have 
\begin{equation}
\|\boldsymbol p-\boldsymbol{\tilde p}_h\|_{0, h,\rho}+\|u-u_h\|_{1,h,\rho}\lesssim \inf\limits_{\boldsymbol{\tilde q}_h\in \boldsymbol{\tilde Q}_h, v_h\in V_h } \Big(\|\boldsymbol p-\boldsymbol{\tilde q}_h\|_{0, h,\rho}+\|u-v_h\|_{1,h,\rho}\Big).
\end{equation}
Hence we need to estimate:
\begin{equation}
 \inf\limits_{\boldsymbol{\tilde q}_h\in \boldsymbol{\tilde Q}_h, v_h\in V_h } \Big(\|\boldsymbol p-\boldsymbol{\tilde q}_h\|_{0, h,\rho}+\|u-v_h\|_{1,h,\rho}\Big).
\end{equation}
Now for any $K\in \mathcal T_h$, we choose $\boldsymbol{\tilde q}_h=(\boldsymbol{q}_h, \boldsymbol{\hat q_h})=(Q_h(\boldsymbol p), \{Q_h(\boldsymbol p)\})$, where $Q_h$ is the local $L^2$ projection from $L^2(\Omega)$ to $\boldsymbol Q_h$.
By the approximation property of the $L^2$ projection, trace 
inequality and noting that $0<\rho\le 1$, we obtain:
\begin{equation}\label{Error:uniform_WG:Primal:p}
\begin{split}
\|\boldsymbol p-\boldsymbol{\tilde q}_h\|^2_{0, h,\rho}&=(c(\boldsymbol p-Q_h(\boldsymbol p)),\boldsymbol p- Q_h(\boldsymbol p))_{\mathcal T_h}+\rho \sum\limits_{K\in \mathcal T_h}h_K\|(\boldsymbol p- Q_h(\boldsymbol p)-(\boldsymbol p- \{Q_h(\boldsymbol p)\}))\cdot\boldsymbol n_K\|^2_{0,\partial K}\\
&\lesssim \|\boldsymbol p-Q_h(\boldsymbol p)\|^2_{\mathcal T_h}
+\sum\limits_{K\in \mathcal T_h}h_K\big(\|(\boldsymbol p- Q_h(\boldsymbol p))\cdot\boldsymbol n_K\|^2_{0,\partial K}+\|(\boldsymbol p- \{Q_h(\boldsymbol p)\})\cdot\boldsymbol n_K\|^2_{0,\partial K}\big)\\
&\lesssim \|\boldsymbol p-Q_h(\boldsymbol p)\|^2_{\mathcal T_h}
+\sum\limits_{K\in \mathcal T_h}\big(\|\boldsymbol p- Q_h(\boldsymbol p)\|^2_{0,K}+h_K^2\|\nabla(\boldsymbol p- Q_h(\boldsymbol p))\|^2_{0,K}\big)\\
&\lesssim h^{2k+2}|\boldsymbol p|^2_{k+1}.
\end{split}
\end{equation}
Next we choose $v_h=\pi_h u$, where $\pi_h$ is the interpolation of $u$ to the continuous finite element 
space, namely, $v_h\in V_h\cap H^1_0(\Omega)$, 
we immediately have:
\begin{equation}\label{Error:uniform_WG:Primal:u}
\begin{split}
\|u-{v}_h\|^2_{1,h,\rho}&=\|\nabla_h(u-v_h)\|^2+\rho^{-1}\sum_{e\in \mathcal E_h}h_e^{-1} \|\hat Q_e([u-v_h])\|^2_{0,e}\\
&=\|\nabla_h(u-v_h)\|^2=\|\nabla_h(u-\pi_h u)\|^2
\lesssim h^{2k+2}|u|^2_{k+2}.
\end{split}
\end{equation}  
Combining \eqref{Error:uniform_WG:Primal:p} and \eqref{Error:uniform_WG:Primal:u}, we get the desired result.
\end{proof}
\begin{theorem}\label{WG-order:mixed}
Let $(\boldsymbol p,  u)\in H({\rm div}, \Omega)\times L^2(\Omega)$ be the solution of \eqref{H1} and $\boldsymbol p\in \boldsymbol {H^{k+1}(\Omega)}, {\rm div} \boldsymbol p\in  H^{k+1}(\Omega), u\in H^{k+1}(\Omega) $, and $(\boldsymbol{\tilde p}_h,  {u}_h)\in \boldsymbol {\tilde Q}_h\times {V}_h$ be the solution of \eqref{Stabilizedpweak_WG} with $\eta=\rho^{-1} h^{-1}_K$. If we choose the spaces ${V}_h\times\boldsymbol {Q}_h\times\hat Q_h={V}_h^k\times\boldsymbol {Q}_h^{k,RT}\times\hat Q_h^{k} $, then for any $0<\rho\leq 1$ the following estimate holds:
\begin{equation}
 \|\boldsymbol p-\boldsymbol {\tilde p}_h\|_{\widetilde{\rm div},h,\rho}+\|u-{u}_h\|\leq C_{r,4} h^{k+1} (|\boldsymbol p|_{k+1}+|{\rm div}\boldsymbol p|_{k+1}+|u|_{k+1}).
\end{equation}
where $ C_{r,4}$ is independent of $h$ and $\rho$.
\end{theorem}
\begin{proof}
From part $2$ of Theorem \ref{wellposed:Aw:graddiv} and Theorem \ref{uniform:error:Au:th}, we have 
\begin{equation}
\|\boldsymbol p-\boldsymbol{\tilde p}_h\|_{{{\rm div}, h,\rho}}+\|u-u_h\| \lesssim \inf\limits_{\boldsymbol{\tilde q}_h\in \boldsymbol{\tilde Q}_h, v_h\in V_h } \Big(\|\boldsymbol p-\boldsymbol{\tilde q}_h\|_{{{\rm div}, h,\rho}}+\|u-v_h\|\Big).
\end{equation}
We need to estimate:
\begin{equation}
 \inf\limits_{\boldsymbol{\tilde q}_h\in \boldsymbol{\tilde Q}_h, v_h\in V_h } \Big(\|\boldsymbol p-\boldsymbol{\tilde q}_h\|_{{{\rm div}, h,\rho}}+\|u-v_h\|\Big).
\end{equation}
Now we choose $\boldsymbol q_h=\pi^{\rm div}_h\boldsymbol p\in \boldsymbol Q_h$, where $\pi^{\rm div}_h$ is the interpolation of $\boldsymbol p$ into the $H({\rm div})$-conforming $RT$ finite element space, namely $\boldsymbol q_h\in \boldsymbol Q_h\cap H({\rm div}, \Omega)$. Since $\boldsymbol q_h\cdot \boldsymbol n$ is single-valued, we can choose $\boldsymbol {\hat q}_h=(\boldsymbol q_h\cdot \boldsymbol n) \boldsymbol n$, then by the approximation property of the $RT$ finite element space, we obtain
\begin{equation}\label{Error:uniform_WG:mixed:p}
\begin{split}
\|\boldsymbol p-\boldsymbol{\tilde q}_h\|_{{\rm div}, h,\rho}=&(c (\boldsymbol{p}-  \boldsymbol{q}_h), \boldsymbol{p}- \boldsymbol{q}_h)_{\mathcal T_h}+( {\rm div} (\boldsymbol{p}-\boldsymbol{q}_h), {\rm div}( \boldsymbol{p}- \boldsymbol{q}_h))_{\mathcal T_h}\\
&+\sum\limits_{K\in \mathcal T_h}\rho^{-1}h_K^{-1}\|(\boldsymbol p-\boldsymbol p_h-(\boldsymbol p-\boldsymbol {\hat p_h}))\cdot \boldsymbol n_K\|^2_{0,\partial K}\\
= &(c ( \boldsymbol{p}-  \boldsymbol{q}_h), \boldsymbol{p}- \boldsymbol{q}_h)_{\mathcal T_h}+( {\rm div} (\boldsymbol{p}-\boldsymbol{q}_h), {\rm div}( \boldsymbol{p}- \boldsymbol{q}_h))_{\mathcal T_h}\\
\lesssim & h^{2k+2}(|\boldsymbol p |^2_{k+1}+|{\rm div}\boldsymbol p|^2_{k+1}).
\end{split}
\end{equation}
Next, we choose $v_h=Q_h(u)$, where $Q_h$ is the $L^2$projection from $L^2(\Omega)$ to $V_h$, and 
we immediately have:
\begin{equation}\label{Error:uniform_WG:mixed:u}
\|u-{v}_h\|\lesssim h^{k+1}|u|_{k+1}.
\end{equation}  
Combining \eqref{Error:uniform_WG:mixed:p} and \eqref{Error:uniform_WG:mixed:u}, we get the desired result.
\end{proof}

 \begin{remark}
    
We must point out that the error estimates obtained here are uniform with respect to the parameter $\rho$. Namely, all 
the constants $C_{r,1},C_{r,2},C_{r,3}$, and $C_{r,4}$ are independent of $\rho$. We figure out the following 
Table \ref{WG:ModiHDG:convergence}.
\end{remark}
\begin{table}[!htbp]
\centering
\begin{tabular}{|c|c|c|c|c|}
 \hline
FEM
& norms  & $\boldsymbol{{U}_h}$ &parameter $\theta$&order\\
  \hline
\multirow{2}{*}{ HDG }
&
 $
\begin{array}{l}\displaystyle
\|\boldsymbol {p}_h\|_{{\rm div},\rho,h}~ \|\tilde u_h\|_{0,\rho,h}~\eqref{HDG:div:norm}
\end{array}
$  
 &$V_h^k\times \boldsymbol Q_h^{k,RT}\times {\hat V}_h^k$ & $\theta=\tau=\rho h_K$&$k+1$
\\
\cline{2-5}
&
$
\begin{array}{l}\displaystyle
\|\boldsymbol {p}_h\|~ \|\tilde u_h\|_{\tilde 1,\rho,h}~\eqref{HDG:grad:norm}

\end{array}
$  
 &$V_h^{k+1}\times \boldsymbol Q_h^k\times {\hat V}_h^k$ & $\theta=\tau=\rho^{-1} h^{-1}_K$&$k+1$
 \\ 
\hline
\multirow{2}{*}{ WG }
& 
$
\begin{array}{l}\displaystyle
\|\boldsymbol{\tilde p}_h\|_{0,h,\rho}~\|u_h\|_{1,h,\rho}~ \eqref{WG:grad:norm} 
\end{array}
$
&$V_h^{k+1}\times \boldsymbol Q_h^k\times {\hat Q}_h^k$ & $\theta=\eta=\rho h_K$&$k+1$
\\
\cline{2-5}
& 
$
\begin{array}{l}\displaystyle
\|\boldsymbol{\tilde p}_h\|_{\widetilde{\rm div},h,\rho}~\|u_h\|~\eqref{WG:div:norm}

\end{array}
$  
&$V_h^k\times \boldsymbol Q_h^{k,RT}\times  {\hat Q}_h^{k}$& $\theta=\eta=\rho^{-1} h^{-1}_K$&$k+1$
\\
\hline
\end{tabular}
\caption{Convergence of HDG and WG}
\label{WG:ModiHDG:convergence}
\end{table}

\section{Relationships between HDG, primal conforming methods and between WG, mixed conforming method}\label{relationship}

In this section, as an application of the uniform stability results, we shall discuss the relationships between HDG and primal conforming methods and the relationship between WG and mixed conforming method.
The proof for the results of this section are also shown in \cite{qingguounified}. For the convenience of reading and the self-consistency of the paper, we show the proof here again. Further, the numerical results verifying the results of this section can be found in \cite{qingguounified}. 
\subsection{Primal conforming methods as the limiting case of HDG methods}
For a given mesh, consider the $H^1$-conforming subspace $ V^c_h=V_h\cap
H^1_0(\Omega)\subset V_h$, then the primal conforming methods in the
variational form are written as: Find  $(u^c_h, \boldsymbol p^c_h)\in
V^c_h\times \boldsymbol Q_h$ such that  
\begin{equation} \label{primalmethod}
\left\{
\begin{aligned}
(c\boldsymbol p^c_h, \boldsymbol q_h)_{\mathcal T_h}+ (\nabla
u^c_h,\boldsymbol q_h)_{\mathcal T_h} &= (\boldsymbol g_1, \boldsymbol
  q_h)_{\mathcal T_h} + \langle g_2, \boldsymbol q_h\cdot
\vect{n}\rangle_{\partial \mathcal T_h}\qquad
\forall \boldsymbol q_h\in \boldsymbol Q_h,\\
-(\boldsymbol p^c_h, \nabla v^c_h)_{\mathcal T_h} &=
(f,v^c_h)_{\mathcal T_h}\qquad\qquad\qquad\qquad\qquad \forall
v^c_h\in V^c_h,
\end{aligned}
\right.
\end{equation}
where $\boldsymbol g_1=0$ and $g_2 = 0$ when applied to the Poisson
equation \eqref{H1}.

We try to prove that the HDG methods
\eqref{Stabilizedpweak} with the stabilization parameter
$\tau = \rho^{-1}h_K^{-1}$ 
converge to primal conforming methods \eqref{primalmethod}
when $\rho \to 0$. 

 First, by $\nabla  V^c_h\subset \nabla_h
V_h\subset \boldsymbol Q_h$, the well-posedness of the primal conforming methods
(cf.  \cite{brenner2007mathematical}) implies that  
\begin{equation} \label{primalstability}
 \|\boldsymbol{p}^c_h\| + \|{u}^c_h\|_{1}  \leq C_p\left( \|f\|_{-1,h}
     + \sup_{\boldsymbol q_h \in \boldsymbol Q_h} \frac{(\boldsymbol
       g_1, \boldsymbol q_h)_{\mathcal T_h} + \langle g_2, \boldsymbol
     q_h\cdot \boldsymbol n\rangle_{\partial \mathcal
     T_h}}{\|\boldsymbol q_h\|} \right),
\end{equation}
where $\|f\|_{-1,h}=\sup\limits_{v^c_h\neq 0,v^c_h\in
V_h^c}\frac{(f,v^c_h)_{\mathcal T_h}}{\|{v}^c_h\|_{1}}$.

Recall that the space define on $\mathcal E_h$ (see
\eqref{Edge:spaces}) of HDG methods is given by  
$$ 
\hat{V}_h = \{\hat{v}_h: \hat{v}_h|_e\in \hat{V}(e), \forall e
\in \mathcal E_h^i, \hat{v}_h|_{\mathcal E_h^\partial} = 0\}.
$$ 
We make the following assumption on the finite element spaces of stabilized hybrid mixed
methods. 

\begin{assumption} \label{Ass:HDG:Primal} 
Assume that the spaces $\boldsymbol{Q}_h, V_h$ and $\hat V_h$
satisfy
\begin{enumerate}
\item $\nabla_h V_h \subset \boldsymbol{Q}_h$;
\item $\{V_h\}|_e \subset \hat V(e)$, ~$\forall e \in \mathcal E_h^i$;   
\item There exists a constant $C_p^I$ independent of $h$, such that
for any $u_h\in V_h$, 
\begin{equation} \label{jump}
\inf_{u_h^I \in V_h^c} \left( \|(u_h^I - u_h^c\| + \|\nabla_h
      (u_h^I-u_h^c)\| \right) \leq C_p^I
\sum_{e\in \mathcal{E}_h} h^{-1/2}_e \|\llbracket u_h \rrbracket
\|_{0,e},
\end{equation}
where $V_h^c = V_h \cap H^1(\Omega)$.
\end{enumerate}
\end{assumption}

We note that the first assumption in Assumption \ref{Ass:HDG:Primal}
ensures the well-posedness of the primal conforming methods \eqref{primalmethod}.
The following example satisfies Assumption \ref{Ass:HDG:Primal}
(see the conforming relatives in \cite{brenner2005c,
brenner2007mathematical}).

\begin{example} \label{Example:HDG:primal}
$\boldsymbol{Q}_h = \boldsymbol{Q}_h^{k}, V_h= V_h^{k+1},
  \hat{V}(e) = \mathcal{P}_{k+1}(e)$, for $k \geq 0$. 
\end{example}

For any given $\tau=\rho^{-1}h_K^{-1}$, we rewrite the HDG methods \eqref{Stabilizedpweak} in the
variational form as: Find $(\boldsymbol{p}_h^{\tau},u_h^{\tau}, \hat
u_h^{\tau})\in \boldsymbol Q_h\times V_h\times \hat{V}_h$ such
that for any $(\boldsymbol q_h, v_h, \hat{v}_h) \in \boldsymbol{Q}_h
\times V_h \times \hat{V}_h$
\begin{equation} \label{newHDG}
\left\{
\begin{aligned}
(c \boldsymbol{p}_h^{\tau}, \boldsymbol{q}_h)_{\mathcal T_h} -
(u_h^\tau, {\rm div}\boldsymbol{q}_h)_{\mathcal T_h} + \langle
\hat{u}_h^\tau ,\boldsymbol q_h \cdot \vect{n}\rangle_{\partial
  \mathcal T_h} &= 0, \\
({\rm div}\boldsymbol p_h^\tau, v_h)_{\mathcal T_h} - \langle
\boldsymbol{p}_h^\tau \cdot \vect{n}, \hat{v}_h\rangle_{\partial \mathcal
  T_h} + \rho^{-1} \langle h_K^{-1} (u_h^\tau - \hat{u}_h^\tau), v_h -
\hat{v}_h \rangle_{\partial \mathcal T_h} &= (f, v_h)_{\mathcal
T_h}.
\end{aligned}
\right.
\end{equation}

\begin{theorem} \label{thm:HDG-primal}
Under the Assumption \ref{Ass:HDG:Primal}, the HDG methods
\eqref{Stabilizedpweak} with $\tau = \rho^{-1}h_K^{-1}$ converge to the primal conforming methods
\eqref{primalmethod} as $\rho \to 0$.
More precisely, we have  
\begin{equation} \label{equ:HDG2primal}
\|\boldsymbol{p}^{\tau}_h-\boldsymbol{p}^{c}_h\|+\|u^{\tau}_h-u_h^c\|_{1,h}
\leq C_{d,3}\rho^{1/2}\|f\|_{-\tilde 1,\rho, h},
\end{equation}
where $C_{d,3}$ is independent of both mesh size $h$ and $\rho$, and
$\|f\|_{-\tilde 1,\rho,h} = \sup\limits_{\tilde{v}_h \in \tilde{V}_h} \frac{(f,
v_h)_{\mathcal{T}_h}}{\|\tilde{v}_h\|_{\tilde{1},\rho,h}}$.
\end{theorem}
\begin{proof}
From the assumption $\{V_h\}|_e \subset \hat V(e)$, by taking $v_h
= v_h^c$ and $\hat{v}_h|_e = v_h^c|_e$ in \eqref{newHDG} and
integrating by parts, we see that 
\begin{equation} \label{HDG62}
-(\boldsymbol{p}_h^{\tau},\nabla v_h^c)_{\mathcal T_h} =
(f,v_h^c)_{\mathcal T_h} \qquad  \forall  v_h^c \in V_h^c.
\end{equation}
Subtracting \eqref{primalmethod} from the first equation of
\eqref{newHDG} and \eqref{HDG62}, we have 
\begin{equation}\label{eq:error}
\left\{
\begin{aligned}
(c(\boldsymbol p_h^{\tau}-\boldsymbol p^c_h),\boldsymbol
q_h)_{\mathcal T_h} + (\nabla u_h^{\tau} - \nabla u^c_h,\boldsymbol
  q_h)_{\mathcal T_h} &= \langle u_h^{\tau}-\hat u_h^{\tau},\boldsymbol
q_h\cdot \boldsymbol n\rangle_{\partial {\mathcal T_h}} \qquad
  \forall \boldsymbol q_h\in \boldsymbol Q_h, \\
-(\boldsymbol p_h^{\tau}-\boldsymbol p_h^c,\nabla v_h^c)_{\mathcal T_h}
&= 0  \qquad\qquad\qquad\qquad \qquad\forall v_h^c\in V_h^c.
\end{aligned}
\right.
\end{equation}
Again, for any $u_h^I \in V_h^c$, we have 
\begin{equation}\label{error}
\left\{
\begin{aligned}
(c(\boldsymbol p_h^{\tau}-\boldsymbol p^c_h), \boldsymbol
 q_h)_{\mathcal T_h} + (\nabla u_h^{I}-\nabla u^c_h,\boldsymbol
   q_h)_{\mathcal T_h} &= \langle u_h^{\tau}-\hat u_h^{\tau},\boldsymbol
 q_h\cdot \boldsymbol n\rangle_{\partial {\mathcal T_h}} + (\nabla
u_h^I - \nabla u_h^{\tau},\boldsymbol q_h)_{\mathcal T_h} \qquad
 \forall \boldsymbol q_h \in \boldsymbol Q_h, \\
-(\boldsymbol p_h^{\tau}-\boldsymbol p_h^c,\nabla v_h^c)_{\mathcal
  T_h} &= 0 \qquad \qquad \qquad \qquad \qquad \qquad\qquad \qquad
  ~~\quad\qquad \forall v_h^c \in V_h^c.
\end{aligned}
\right.
\end{equation}
Because $\boldsymbol p_h^{\tau}-\boldsymbol p^c_h\in \boldsymbol
Q_h$ and $v_h^{c}-u^c_h\in V_h^c$, using \eqref{primalstability},
trace inequality, inverse inequality and Cauchy inequality, we obtain
\begin{equation}\label{error1}
\begin{aligned}
\|\boldsymbol p_h^{\tau}-\boldsymbol p_h^{c}\|+\|u_h^I-u_h^c\|_1 &\leq
C_p\sup_{\boldsymbol q_h\in \boldsymbol Q_h}\frac{\langle
u_h^{\tau}-\hat u_h^{\tau},\boldsymbol q_h\cdot \boldsymbol
n\rangle_{\partial {\mathcal T_h}} + (\nabla u_h^{I}-\nabla
u^\tau_h,\boldsymbol q_h)_{\mathcal T_h}}{\|\boldsymbol q_h\|} \\
& \lesssim |u_h^I - u_h^\tau|_{1,h} + \langle h^{-1}(u_h^\tau -
\hat{u}_h^\tau), u_h^\tau - \hat{u}_h^\tau \rangle_{\partial \mathcal
T_h}^{1/2}. 
\end{aligned}
\end{equation}
Noting that $\{V_h\}|_e
\subset \hat{V}(e)$, and 
\begin{equation}\label{identity}
 \langle u_h^{\tau}-\hat{u}_h^{\tau},
 u_h^{\tau}-\hat{u}_h^{\tau}\rangle_{\partial {\mathcal T_h}}
 =2\langle\{u_h^{\tau}-\hat{u}_h^{\tau}\},
 \{u_h^{\tau}-\hat{u}_h^{\tau}\}\rangle_{\mathcal{E}_h}+
 \frac{1}{2}\langle \lbrack\!\lbrack u_h^{\tau}\rbrack\!\rbrack,
 \lbrack\!\lbrack u_h^{\tau} \rbrack\!\rbrack\rangle_{\mathcal{E}_h}.
\end{equation}
Therefore, Assumption \ref{Ass:HDG:Primal}, \eqref{error1}, and
\eqref{identity} imply that 
$$ 
\begin{aligned}
\|\boldsymbol p_h^{\tau}-\boldsymbol p_h^{c}\| +
\|u_h^{\tau}-u_h^c\|_{1,h} & \leq \inf_{u_h^I \in V_h^c} 
\left( \|\boldsymbol p_h^\tau - \boldsymbol p_h^c\| + \|u_h^I -
u_h^c\|_{1} + \|u_h^\tau - u_h^I\|_{1,h} \right) \\
& \lesssim  \langle h^{-1}(u_h^\tau - \hat{u}_h^\tau), u_h^\tau -
\hat{u}_h^\tau \rangle_{\partial \mathcal T_h}^{1/2} +  \inf_{u_h^I \in
  V_h^c} \|u_h^\tau - u_h^I\|_{1,h} \\
& \lesssim  \langle h^{-1}(u_h^\tau - \hat{u}_h^\tau), u_h^\tau -
\hat{u}_h^\tau \rangle_{\partial \mathcal T_h}^{1/2} + \sum_{e\in \mathcal
  E_h} h_e^{-1/2} \|[u_h]\|_{0,e} \\
& \lesssim \langle h^{-1}(u_h^\tau - \hat{u}_h^\tau), u_h^\tau -
\hat{u}_h^\tau \rangle_{\partial \mathcal T_h}^{1/2} \\
& \lesssim \rho^{1/2}\|f\|_{-\tilde 1, \rho, h},
\end{aligned}
$$
where Corollary \ref{uniform_stable} was used in the last step. 
\end{proof}

\begin{remark}
From the definition of
$\|\cdot\|_{\tilde{1},\rho, h}$, when $\rho \lesssim 1$, we have  
$$ 
\begin{aligned}
\inf_{\hat{v}_h \in \hat{V}_h} \|\tilde{v}_h\|^2_{\tilde{1},
  \rho, h} &= \inf_{\hat{v}_h \in \hat{V}_h} (\nabla v_h, \nabla
v_h)_{\mathcal T_h} + \sum_{K\in \mathcal T_h} \rho^{-1} h_K^{-1}
\langle v_h - \hat{v}_h, v_h - \hat{v}_h \rangle_{\partial K} \\
& \simeq (\nabla v_h, \nabla v_h)_{\mathcal T_h} + \rho^{-1} 
\sum_{e\in \mathcal E_h}h_e^{-1} \|\llbracket v_h \rrbracket\|_{0,e}^2
\gtrsim \|v_h\|_{1,h}. 
\end{aligned}
$$ 
Hence, when $\rho \lesssim 1$,  
$$ 
\|f\|_{-\tilde 1,\rho,h} = \sup_{\tilde{v}_h \in \tilde{V}_h} \frac{(f,
    v_h)_{\mathcal T_h}}{\|\tilde{v}_h\|_{\tilde{1},\rho, h}} = 
\sup_{v_h \in V_h} \frac{(f, v_h)_{\mathcal
  T_h}}{\inf_{\hat{v}_h \in
    \hat{V}_h}\|\tilde{v}_h\|_{\tilde{1},\rho, h}} \lesssim
    \sup_{v_h \in V_h} \frac{(f, v_h)_{\mathcal T_h}}{\|v_h\|_{1,h}}
    \lesssim \|f\|, 
$$ 
which means that the solutions of HDG methods converge to those of
primal conforming methods with order $\rho^{1/2}$ at least.
\end{remark}

\subsection{Mixed conforming methods as the limiting case of WG methods}
For a given mesh, consider the $\boldsymbol H({\rm div})$-conforming subspace
$\boldsymbol Q^c_h := \boldsymbol Q_h\cap \boldsymbol H({\rm div},
\Omega)\subset \boldsymbol Q_h$, the mixed conforming methods in variational form are written as: Find
$(\boldsymbol p^c_h, u^c_h) \in \boldsymbol Q^c_h \times V_h$ such
that  
\begin{equation}\label{mixedmethod}
\left\{
\begin{aligned}
(c\boldsymbol p^c_h,\boldsymbol q^c_h)_{\mathcal T_h} - (u^c_h,{\rm
div} \boldsymbol q^c_h)_{\mathcal T_h} &= (g_1, \boldsymbol
  q_h^c)_{\mathcal T_h}
~\qquad \qquad \qquad \quad \forall \boldsymbol q^c_h\in \boldsymbol Q^c_h,\\
 ({\rm div} \boldsymbol p^c_h, v_h)_{\mathcal T_h}&= (f,v_h)_{\mathcal
   T_h} + \langle g_2, v_h \rangle_{\partial \mathcal T_h} \qquad
   \forall v_h\in V_h,
\end{aligned}
\right.
\end{equation}
where $g_1 = 0$ and $g_2 = 0$ when applied to the Poisson equation
\eqref{pois}.

We will now try to prove that WG methods
\eqref{Stabilizedpweak_WG} with $\eta =
\rho^{-1}h_K^{-1}$  converge to mixed conforming methods \eqref{mixedmethod} when
$\rho \to 0$.

First, by $V_h\subset {\rm div} \boldsymbol Q^c_h\subset
{\rm div}_h \boldsymbol Q_h\subset V_h$, the well-posedness of the
mixed conforming methods (cf. \cite{brezzi1991mixed, boffi2013mixed}) implies
that  
\begin{equation} \label{equ:well-posedness-mixed}
\|\boldsymbol p_h^c\|_{\boldsymbol H({\rm div})} + \|v_h^c\| \leq C_M
\left( 
\|f\| + \sup_{\boldsymbol q_h^c \in \boldsymbol Q_h^c} \frac{(g_1,
  \boldsymbol q_h^c)_{\mathcal T_h}}{\|\boldsymbol
q_h^c\|_{\boldsymbol H({\rm div})}} + \sup_{\boldsymbol v_h \in V_h}
\frac{\langle g_2, v_h \rangle_{\partial \mathcal T_h}}{\|v_h\|}
\right).
\end{equation}

Recall that the spaces defined on $\mathcal E_h$ (see
\eqref{Edge:spaces}) of WG methods are given by 
$$ 
\hat{\boldsymbol Q}_h = \{\hat{\boldsymbol p}_h: \hat{\boldsymbol
p}_h|_e \in \hat{Q}(e)\vect{n}_e, \forall e\in \mathcal{E}_h\},
  \qquad \hat{Q}_h = \{\hat{p}_h: \hat{p}_h|_{e}, e\in
  \hat{Q}(e), \forall e\in \mathcal{E}_h\}.
$$ 
We make the following assumption on the finite element spaces of WG
methods.

\begin{assumption}\label{Ass:WG:mixed}
Assume that the spaces $\boldsymbol{Q}_h$, $\hat{\boldsymbol Q}_h$
and $V_h$ satisfy
\begin{enumerate}
\item ${\rm div}_h \boldsymbol{Q}_h = V_h$;
\item $\{\!\!\{\boldsymbol{Q}_h\}\!\!\}|_e \subset \hat Q(e),
~\forall e \in \mathcal{E}_h$;   
\item There exists a constant $C_{M}^I$ independent of $h$, such that
for any $\boldsymbol{p}_h \in \boldsymbol{Q}_h$, 
\begin{equation}\label{approx_mixed}
\inf_{\boldsymbol{p}^{I}_h\in
\boldsymbol{Q}^{c}_h} (\|\boldsymbol{p}^{I}_h-\boldsymbol{p}_h\|+\|{\rm
div}_h(\boldsymbol{p}^{I}_h-\boldsymbol{p}_h)\|) \leq
C_{M}^I \sum_{e\in
\mathcal{E}^i_h} h_e^{-1/2}\|[\boldsymbol{p}_h]\|_{0,e},
\end{equation}
where $\boldsymbol Q^c_h = \boldsymbol Q_h \cap \boldsymbol H({\rm
div}; \Omega)$.
\end{enumerate}
\end{assumption}

We note that the first assumption in Assumption \ref{Ass:WG:mixed}
ensures well-posedness of the mixed conforming methods \eqref{mixedmethod}.
Several examples are given below.

\begin{example}\label{Example:WG:RT}
Raviart-Thomas type: $ \boldsymbol{Q}_h = \boldsymbol{Q}_h^{k, RT},
\hat Q(e) = \mathcal{P}_k(e), V_h = V_h^k$, for $k \geq 0$. 
\end{example}

\begin{example}\label{Example:WG:BDM}
Brezzi-Douglas-Marini type: $
\boldsymbol{Q}_h = \boldsymbol Q_h^{k+1}, \hat Q(e) =
\mathcal{P}_{k+1}(e), V_h = V_h^k$, for $k\geq 0$. 
\end{example}

\begin{lemma}
If we choose the spaces as in Example \ref{Example:WG:BDM} or Example
\ref{Example:WG:RT}, then Assumption \ref{Ass:WG:mixed} holds. 
\end{lemma}

\begin{proof}
We only sketch the proof of \eqref{approx_mixed} in Assumption
\ref{Ass:WG:mixed}. Denote the set of degrees of freedom of RT or BDM
element by $D$, see \cite{brezzi1991mixed, boffi2013mixed}. We then
define $\boldsymbol p_h^I$ as 
$$ 
d(\boldsymbol p_h^I) = \frac{1}{|\mathcal T_d|} \sum_{K \in \mathcal
  T_d} d(\boldsymbol p_h|_T) \qquad \forall d \in D,
$$ 
where $\mathcal T_d$ denotes the set of elements that share the degrees
of freedom $d$ and $|\mathcal T_d|$ denotes the cardinality of this
set. By the standard scaling argument, 
$$ 
\sum_{K \in \mathcal T_h} \|\boldsymbol p_h^I - \boldsymbol p_h\|
\lesssim \sum_{e\in \mathcal E_h^i} h_e^{1/2} \|[\boldsymbol
p_h]\|_{0,e}.
$$ 
Then \eqref{approx_mixed} follows from the inverse inequality.
\end{proof}

For any given $\eta=\rho^{-1}h_K^{-1}$, we rewrite the WG methods \eqref{Stabilizedpweak_WG} in the variational
form as: Find $(\boldsymbol{p}^{\eta}_h, u^{\eta}_h, \hat{\boldsymbol
p}^{\eta}_h)\in \boldsymbol Q_h\times V_h\times \hat{\boldsymbol
Q}_h$ such that for any  $(\boldsymbol{q}_h, v_h, \hat{\boldsymbol
q}_h)\in \boldsymbol Q_h\times V_h\times \hat{\boldsymbol Q}_h$
\begin{equation} \label{newWG}
\left\{
\begin{aligned}
(c \boldsymbol{p}^{\eta}_h, \boldsymbol{q}_h)_{\mathcal T_h} +
\rho^{-1} \langle h_K^{-1} (\boldsymbol p_h^\eta - \hat{\boldsymbol
    p}_h) \cdot \vect{n}, (\boldsymbol q_h - \hat{\boldsymbol
    q}_h) \cdot \vect{n} \rangle_{\partial \mathcal T_h} + 
(\nabla u_h, q_h)_{\mathcal T_h} - \langle u_h, \hat{\boldsymbol q}_h
\cdot \vect{n} \rangle_{\partial \mathcal T_h} & = 0,\\
-(\boldsymbol p_h^\eta, \nabla v_h)_{\partial \mathcal T_h} + \langle
\hat{\boldsymbol p}_h^\eta \cdot \vect{n}, v_h \rangle_{\partial
  \mathcal{T}_h}& = (f, v_h).
\end{aligned}
\right.
\end{equation}

\begin{theorem} \label{thm:WG-mixed}
Under the Assumption \ref{Ass:WG:mixed}, WG WG methodsmethods
\eqref{Stabilizedpweak_WG} converge to the mixed conforming methods
\eqref{mixedmethod} as $\rho \to 0$ with $\eta =
\rho^{-1}h_K^{-1}$.  More precisely, we have  
\begin{equation} \label{equ:WG-mixed}
\|\boldsymbol{p}^{\eta}_h-\boldsymbol{p}^{c}_h\|_{\boldsymbol H_h({\rm
  div})}+\|u^{\eta}_h-u_h^c\| \leq C_{w,3}\rho^{1/2}\|f\|,
\end{equation}
where $C_{w,3}$ is independent of both mesh size $h$ and $\rho$.
\end{theorem}

\begin{proof}
From the assumption $\{\!\!\{\boldsymbol Q_h\}\!\!\}|_e \subset
\hat{Q}(e)$, by taking $\boldsymbol q_h = \boldsymbol q_h^c$ and
$\hat{\boldsymbol q}_h|_e = (\boldsymbol q_h^c \cdot
\vect{n}_e)\vect{n}_e$ in \eqref{newWG} and integrating by parts, we
see that $(\boldsymbol{p}_h^{\eta}, u_h^{\eta})$ satisfies 
\begin{equation} \label{equ:WG-conforming}
(c \boldsymbol{p}^{\eta}_h, \boldsymbol{q}^c_h)_{\mathcal
T_h}-(u^{\eta}_h,  {\rm div} \boldsymbol{q}^c_h)_{\mathcal T_h}=0
\qquad \forall \boldsymbol{q}^c_h\in \boldsymbol Q^c_h.
\end{equation}
Subtracting \eqref{mixedmethod} from \eqref{equ:WG-conforming} and
the second equation of \eqref{newWG}, we have 
\begin{equation}\label{mixederror}
\left\{
\begin{aligned}
(c (\boldsymbol{p}^{\eta}_h-\boldsymbol{p}^{c}_h),
\boldsymbol{q}^c_h)_{\mathcal T_h} 
- (u^{\eta}_h-u_h^c,  {\rm div} \boldsymbol{q}^c_h)_{\mathcal
T_h} &= 0 \qquad \qquad\qquad\qquad\qquad ~~\forall \boldsymbol{q}^c_h\in
\boldsymbol Q^c_h,\\
({\rm div}(\boldsymbol{p}^{\eta}_h - \boldsymbol{p}^{c}_h),
 v_h)_{\mathcal T_h} &=
\langle(\boldsymbol{p}^{\eta}_h-\hat{\boldsymbol p}^{\eta}_h)\cdot
\boldsymbol {n}, v_h\rangle_{\partial {\mathcal T_h}} \qquad  \forall
v_h \in V_h. 
\end{aligned}
\right.
\end{equation}
Noting that $\boldsymbol p_h^\eta \not\in \boldsymbol{Q}_h^c$, we have
that, for any  $\boldsymbol{p}^{I}_h\in \boldsymbol Q^c_h$, 
\begin{equation}\label{mixederror1}
\left\{
\begin{aligned}
&(c (\boldsymbol{p}^{I}_h-\boldsymbol{p}^{c}_h),
\boldsymbol{q}^c_h)_{\mathcal T_h} - (u^{\eta}_h-u_h^c,  {\rm div}
\boldsymbol{q}^c_h)_{\mathcal T_h} =
(c(\boldsymbol{p}^{I}_h-\boldsymbol{p}^{\eta}_h),\boldsymbol{q}^c_h)_{\mathcal
T_h}~~~\qquad\qquad\qquad \forall \boldsymbol{q}^c_h\in
\boldsymbol Q^c_h,\\
&({\rm div} (\boldsymbol{p}^{I}_h-\boldsymbol{p}^{c}_h),v_h)_{\mathcal
T_h} = \langle(\boldsymbol{p}^{\eta}_h - \hat{\boldsymbol
p}^{\eta}_h)\cdot \boldsymbol{n}, v_h\rangle_{\partial \mathcal T_h} 
+({\rm div}(\boldsymbol{p}^{I}_h-\boldsymbol{p}^{\eta}_h),v_h)_{\mathcal
  T_h}\qquad\qquad\forall  v_h \in V_h. 
\end{aligned}
\right.
\end{equation}
Because $(\boldsymbol{p}^{I}_h-\boldsymbol{p}^{c}_h)\in
\boldsymbol Q^c_h, (u^{\eta}_h-u_h^c)\in V_h$, by the well-posedness
of the mixed conforming methods \eqref{equ:well-posedness-mixed}, trace
inequality, inverse inequality and Cauchy inequality, we have
$$ 
\begin{aligned}
& \|\boldsymbol{p}^{I}_h-\boldsymbol{p}^{c}_h\|_{\boldsymbol H({\rm
div}) }+\|u^{\eta}_h-u_h^c\| \\
\leq &~ C_M\left( \sup_{\boldsymbol{q}^c_h\in
\boldsymbol{Q}^c_h} \frac{(c(\boldsymbol{p}^{I}_h -
\boldsymbol{p}^{\eta}_h), \boldsymbol{q}^c_h)_{\mathcal
T_h}}{\|\boldsymbol{q}^c_h\|_{\boldsymbol H({\rm div})}} 
+\sup_{v_h\in V_h} \frac{\langle(\boldsymbol{p}^{\eta}_h - \hat
{\boldsymbol p}^{\eta}_h)\cdot \boldsymbol {n}, v_h\rangle_{\partial
{\mathcal T_h}} + ({\rm
div}(\boldsymbol{p}^{I}_h-\boldsymbol{p}^{\eta}_h),v_h)_{\mathcal
T_h}}{\|v_h\|} \right)\\
\lesssim &~ 
\|\boldsymbol{p}^{I}_h-\boldsymbol{p}^{\eta}_h\|+\|{\rm
div}_h(\boldsymbol{p}^{I}_h-\boldsymbol{p}^{\eta}_h)\|+
\langle h^{-1}(\boldsymbol{p}^{\eta}_h-\hat {\boldsymbol
p}^{\eta}_h)\cdot \boldsymbol {n},
(\boldsymbol{p}^{\eta}_h-\hat{\boldsymbol p}^{\eta}_h)\cdot
\boldsymbol {n}\rangle_{\partial \mathcal T_h}^{1/2}.
\end{aligned}
$$
Hence, by Assumption \ref{Ass:WG:mixed} and inverse inequality, we have
$$ 
\begin{aligned}
\|\boldsymbol{p}^{\eta}_h-\boldsymbol{p}^{c}_h\|_{\boldsymbol H({\rm
    div})}+\|u^{\eta}_h-u_h^c\| 
& \lesssim \langle h^{-1}(\boldsymbol{p}^{\eta}_h-\hat {\boldsymbol
p}^{\eta}_h)\cdot \boldsymbol {n},
(\boldsymbol{p}^{\eta}_h-\hat{\boldsymbol p}^{\eta}_h)\cdot
\boldsymbol {n}\rangle_{\partial \mathcal T_h}^{1/2}
+ \inf_{\boldsymbol p_h^I \in \boldsymbol Q_h^c}
\left( \|\boldsymbol{p}^{I}_h-\boldsymbol{p}^{\eta}_h\| 
+ \|{\rm div}_h(\boldsymbol{p}^{I}_h-\boldsymbol{p}^{\eta}_h)\| 
\right) \\
& \lesssim \langle h^{-1}(\boldsymbol{p}^{\eta}_h-\hat {\boldsymbol
p}^{\eta}_h)\cdot \boldsymbol {n},
(\boldsymbol{p}^{\eta}_h-\hat{\boldsymbol p}^{\eta}_h)\cdot
\boldsymbol {n}\rangle_{\partial \mathcal T_h}^{1/2} + \sum_{e\in
  \mathcal{E}_h^i} h_e^{-1/2}\|[\boldsymbol p_h^\eta]\|_{0,e}.
\end{aligned}
$$
From the fact that 
$$
\langle(\boldsymbol{p}^{\eta}_h-\boldsymbol {\hat p}^{\eta}_h)\cdot
\boldsymbol {n}, (\boldsymbol{p}^{\eta}_h-\boldsymbol {\hat
p}^{\eta}_h)\cdot \boldsymbol {n}\rangle_{\partial {\mathcal T_h}}
= 2\langle \{\!\!\{\boldsymbol{p}^{\eta}_h-\boldsymbol {\hat
  p}^{\eta}_h\}\!\!\},\{\!\!\{\boldsymbol{p}^{\eta}_h-\boldsymbol
{\hat p}^{\eta}_h\}\!\!\}\rangle_{\mathcal E_h}+\frac{1}{2}\langle
[\boldsymbol{p}^{\eta}_h],[\boldsymbol{p}^{\eta}_h]\rangle_{\mathcal
  E_h},
$$
we obtain  
$$ 
\begin{aligned}
\|\boldsymbol{p}^{\eta}_h-\boldsymbol{p}^{c}_h\|_{\boldsymbol H({\rm
div})}+\|u^{\eta}_h-u_h^c\| 
& \lesssim \langle h^{-1}(\boldsymbol{p}^{\eta}_h-\hat {\boldsymbol
p}^{\eta}_h)\cdot \boldsymbol {n},
(\boldsymbol{p}^{\eta}_h-\hat{\boldsymbol p}^{\eta}_h)\cdot
\boldsymbol {n}\rangle_{\partial \mathcal T_h}^{1/2} + \sum_{e\in
  \mathcal{E}_h^i} h_e^{-1/2}\|[\boldsymbol p_h^\eta]\|_{0,e} \\
& \lesssim \langle h^{-1}(\boldsymbol{p}^{\eta}_h-\hat {\boldsymbol
p}^{\eta}_h)\cdot \boldsymbol {n},
(\boldsymbol{p}^{\eta}_h-\hat{\boldsymbol p}^{\eta}_h)\cdot
\boldsymbol {n}\rangle_{\partial \mathcal T_h}^{1/2} \lesssim
\rho^{1/2}\|f\|,
\end{aligned}
$$ 
where we used Corollary \ref{uniform_stable_WG} in the last
step. This completes the proof. 
\end{proof}

\section{Analysis of HDG and WG}\label{sec:HDG:WG}
In this section, we present the analysis of HDG and WG methods and hence prove the uniformly well-posed results of the 
HDG and WG methods provided in Section \ref{sec:framework}. Namely, we prove the Theorem \ref{wellposed:Ah:divgrad} for HDG methods and the Theorem  \ref{wellposed:Aw:graddiv} for WG methods. That means we need to prove that the HDG and WG methods satisfy the consistency, the uniform continuity and the inf-sup condition uniformly with respect to the corresponding norms. 

\begin{lemma}\label{consistent}
Both the HDG methods and WG methods are consistent. 
\end{lemma}
\begin{proof}
By the verification of  \eqref{general: consistence}, the proof is obvious.
\end{proof}
\subsection{Proof for Part  \ref{wellposed:Ah:div} of Theorem \ref{wellposed:Ah:divgrad}} 

\begin{theorem}\label{bounded:ahbhch}
For any $0<\rho\leq 1$, the bilinear form $A_h((\boldsymbol{p}_h, \tilde u_h),(\boldsymbol{q}_h, \tilde v_h))$ is uniformly continuous.
\end{theorem}
\begin{proof}
The boundedness of $a_h(\boldsymbol{p}_h,\boldsymbol{q}_h)$ is obvious. Before we discuss the boundedness of $b_h(\boldsymbol{q}_h, \tilde u_h)$, by \eqref{equ:dg-identity_1} and noting $\lbrack\!\lbrack \hat u_h \rbrack\!\rbrack=0$, we rewrite $b_h(\boldsymbol{q}_h, \tilde u_h)$ as:
\begin{equation}\label{Identity:bh}
b_h(\boldsymbol{q}_h, \tilde u_h)=-(u_h, {\rm div}_h \boldsymbol{q}_h)_{\mathcal T_h}+\sum\limits_{K\in \mathcal{T}_h}\langle\hat{u}_h, \boldsymbol{q}_h\cdot\boldsymbol{n}_K\rangle_{\partial K}=-(u_h, {\rm div}_h \boldsymbol{q}_h)_{\mathcal T_h}+\sum\limits_{e\in \mathcal{E}_h^i}\langle\hat{u}_h,[\boldsymbol{q}_h]\rangle_e.
\end{equation}
Now we show the boundedness of $b_h(\boldsymbol{q}_h, \tilde u_h)$ here. By \eqref{Identity:bh} and the definition of $\hat P_e$, 
we have:
\begin{eqnarray}
\nonumber b_h(\boldsymbol{q}_h, \tilde u_h)
&=&-(u_h, {\rm div}_h \boldsymbol{q}_h)_{\mathcal T_h}+\langle\hat{u}_h,[\boldsymbol{q}_h]\rangle_{\mathcal E^i_h}\\
&=&-(u_h, {\rm div}_h\boldsymbol{q}_h)_{\mathcal T_h}+\langle\hat{u}_h,\hat P_e([\boldsymbol{q}_h])\rangle_{\mathcal E^i_h}\\
&\leq& \|{\rm div}_h \boldsymbol{q}_h\| \|u_h\|+\Big(\rho^{-1}\sum\limits_{e\in \mathcal{E}^i_h}h_e^{-1} \langle\hat P_e([\boldsymbol{q}_h]),\hat P_e([\boldsymbol{q}_h])\rangle_e\Big)^{\frac{1}{2}}  \Big(\rho \sum\limits_{e\in \mathcal{E}^i_h}h_e\langle\hat u_h,\hat u_h\rangle_e\Big)^{\frac{1}{2}} \\
&\leq& \|\boldsymbol{q}_h\|_{{\rm div}, \rho,h} \|\tilde u_h\|_{0,\rho,h},
\end{eqnarray}
which proves the boundedness of $b_h(\boldsymbol{q}_h, \tilde u_h)$.

Next we prove the boundedness of $c_h( \tilde{u}_h, \tilde{v}_h)$. 

By the Cauchy inequality, we have:
\begin{eqnarray}
\nonumber |c_h( \tilde{u}_h, \tilde{v}_h)|&=&\rho\sum\limits_{K\in \mathcal{T}_h}h_K\langle u_h-\hat{u}_h, v_h-\hat{v}_h\rangle_{\partial K}\\
&\leq&  \Big(\rho\sum\limits_{K\in \mathcal{T}_h}h_K\langle u_h-\hat{u}_h, u_h-\hat{u}_h\rangle_{\partial K}\Big)^{\frac{1}{2}}\Big(\rho\sum\limits_{K\in \mathcal{T}_h}h_K\langle v_h-\hat{v}_h, v_h-\hat{v}_h\rangle_{\partial K}\Big)^{\frac{1}{2}}.
\end{eqnarray}
By the trace inequality, inverse inequality, and noting that $0<\rho\leq 1$, we have: 
\begin{eqnarray}
&&\rho\sum\limits_{K\in \mathcal{T}_h}h_K\langle u_h-\hat{u}_h, u_h-\hat{u}_h\rangle_{\partial K}\\
&=&\rho\sum\limits_{K\in \mathcal{T}_h}h_K\Big(\langle u_h, u_h\rangle_{\partial K}-2\langle u_h, \hat{u}_h\rangle_{\partial K}+\langle \hat{u}_h, \hat{u}_h\rangle_{\partial K}\Big)\\
&\leq&2\rho\sum\limits_{K\in \mathcal{T}_h}h_K\Big(\langle u_h, u_h\rangle_{\partial K}+\langle \hat{u}_h, \hat{u}_h\rangle_{\partial K}\Big)\\
&\lesssim&2\rho\sum\limits_{K\in \mathcal{T}_h}h_K\langle u_h, u_h\rangle_{\partial K}+4\rho\sum\limits_{K\in \mathcal{E}_h^i}h_e\langle \hat{u}_h, \hat{u}_h\rangle_{e}\\
&\lesssim&2\sum\limits_{K\in \mathcal{T}_h}(u_h, u_h)_{K}+4\rho\sum\limits_{K\in \mathcal{E}_h^i}h_e\langle \hat{u}_h, \hat{u}_h\rangle_{e}
\leq 4 \|\tilde u_h\|^2_{0,\rho,h}.
\end{eqnarray}
Similarly, we have:
\begin{equation}
\rho\sum\limits_{K\in \mathcal{T}_h}h_K\langle v_h-\hat{v}_h, v_h-\hat{v}_h\rangle_{\partial K}\lesssim \|\tilde v_h\|^2_{0,\rho,h}.
\end{equation}
Hence, we obtain $|c_h( \tilde{u}_h, \tilde{v}_h)|\lesssim \|\tilde u_h\|_{0,\rho,h} \|\tilde v_h\|_{0,\rho,h}$.
\end{proof}

We denote 
\begin{equation}
{\rm Ker}(B)=\{\boldsymbol q_h\in \boldsymbol Q_h: b_h(\boldsymbol q_h,\tilde u_h)=0, \forall \tilde u_h\in \tilde V_h\}.
\end{equation}
Then, we have the coercivity of $a_h(\cdot,\cdot)$ on the ${\rm Ker}(B)$ as follows:
\begin{theorem}\label{coercivity:ah}
Assume that ${\rm div}_h \boldsymbol Q_h \subset V_h$, then
\begin{equation}
a_h(\boldsymbol{p}_h,\boldsymbol{p}_h)\geq \|\boldsymbol{p}_h\|^2_{{\rm div}, \rho,h},~~\forall \boldsymbol{p}_h \in {\rm Ker}(B).
\end{equation}
\end{theorem}
\begin{proof}
Since 
\begin{equation}
{\rm Ker}(B)=\{\boldsymbol q_h\in \boldsymbol Q_h: b_h(\boldsymbol q_h,\tilde u_h)=0,~\forall \tilde u_h\in \tilde V_h\},
\end{equation}
then by \eqref{Identity:bh} and under the assumption that ${\rm div}_h \boldsymbol Q_h \subset V_h$, we have:
\begin{eqnarray}
{\rm Ker}(B)
&=&\{\boldsymbol q_h\in \boldsymbol Q_h:-(u_h, {\rm div} \boldsymbol{q}_h)_{\mathcal T_h}+\langle\hat{u}_h,[\boldsymbol{q}_h]\rangle_{\mathcal E^i_h}=0,~\forall \tilde u_h\in \tilde V_h\}\\
&=&\{\boldsymbol q_h\in \boldsymbol Q_h:-(u_h, {\rm div} \boldsymbol{q}_h)_{\mathcal T_h}+\langle\hat{u}_h,\hat P_e([\boldsymbol{q}_h])\rangle_{\mathcal E^i_h}=0,~\forall \tilde u_h\in \tilde V_h\}\\
&=&\{\boldsymbol q_h\in \boldsymbol Q_h: {\rm div}_h\boldsymbol q_h =0, \hat P_e([\boldsymbol{q}_h])=0\}.
\end{eqnarray}
Hence, by the definition of $\|\boldsymbol q_h\|_{{\rm div},\rho, h}$, we obtain $a_h(\boldsymbol{p}_h,\boldsymbol{p}_h)\geq \|\boldsymbol{p}_h\|^2_{{\rm div},\rho,h},~\forall~ \boldsymbol{p}_h \in {\rm Ker}(B)$.
\end{proof}   

\begin{lemma}\label{RT:jump}
Given the edges (faces) $e_1, e_2, \cdots, e_{d+1}$ of the simplex $K$ and functions $\boldsymbol q\in \boldsymbol{L^2(K)}$ and~$\hat {\zeta}_i\in L^2(e_i), i=1,\cdots,d+1$, there is a unique function $\boldsymbol z\in \boldsymbol{\mathcal P_r(K)} \oplus\boldsymbol x \mathcal P_r(K), r\ge 0$ such that,  
\begin{eqnarray}
&&(\boldsymbol z-\boldsymbol q, \boldsymbol p)_K=0, ~~\forall~ \boldsymbol p\in \boldsymbol{\mathcal P_{r-1}(K)},\\
&&(\boldsymbol z\cdot \boldsymbol n_i-\hat {\zeta}_i, \hat v)_{e_i}=0,~~\forall~ \hat v\in \mathcal P_r(e_i), i=1,\cdots, d+1,
\end{eqnarray}
where $\boldsymbol n_i$ is the outward normal unit vector of $e_i$. Moreover:
\begin{equation}\displaystyle
\|\boldsymbol z\|_{0,K}\leq C_{d,r}\Big(\|\boldsymbol q\|_{0,K}+h_K^{1/2}\sum_{i=1}^{d+1}\|\hat {\zeta}_i\|_{0,e_i}\Big),
\end{equation}
where $C_{d,r}$ depends only on $d,r$, and the shape regular constant. 
\end{lemma}
\begin{proof}
Similar to the definition of the local Raviart-Thomas finite element, the well-posedness of $\boldsymbol z$ is obvious. 
Then, from a simple scaling argument, the estimate is desired.
\end{proof}

Let $\boldsymbol {\mathcal {\tilde P}_{r-1}(K)}=(\mathcal {\tilde P}_{r-1}(K))^d$ be the vector homogeneous polynomials of degree $r-1$. Similar to Lemma \ref{RT:jump}, we also have: 
\begin{lemma}\label{BDM:jump}
Given the edges (faces) $e_1, e_2, \cdots, e_{d+1}$ of the simplex $K$ and functions $\boldsymbol q\in \boldsymbol{L^2(K)}$ and~$\hat {\zeta}_i\in L^2(e_i), i=1,\cdots,d+1$, there is a unique function $\boldsymbol z\in \boldsymbol{\mathcal P_r(K)}, r\ge 1$ such that,  
\begin{eqnarray}
&&(\boldsymbol z-\boldsymbol q, \boldsymbol p)_K=0, ~~\forall~ \boldsymbol p\in \boldsymbol{\mathcal P_{r-2}(K)} \oplus \boldsymbol{\mathcal  S_{r-1}(K)},\\
&&\langle\boldsymbol z\cdot \boldsymbol n_i-\hat {\zeta}_i, \hat v\rangle_{e_i}=0,~~\forall~ \hat v\in \mathcal P_{r}(e_i), i=1,\cdots, d+1,
\end{eqnarray}
\end{lemma}
where $\boldsymbol n_i$ is the outward normal unit vector of $e_i$ and $\boldsymbol{\mathcal  S_{r-1}(K)}= \left\{ \boldsymbol v \in  \boldsymbol {\mathcal {\tilde P}_{r-1}(K)}
 : \boldsymbol x \cdot \boldsymbol v = 0 \right\}$. Moreover:
\begin{equation}\displaystyle
\|\boldsymbol z\|_{0,K}\leq C_{d,r}\Big(\|\boldsymbol q\|_{0,K}+h_K^{1/2}\sum_{i=1}^{d+1}\|\hat {\zeta}_i\|_{0,e_i}\Big),
\end{equation}
where $C_{d,r}$ depends only on $d,r$, and the shape regular constant. 

\begin{proof}
Similar to the definition of the local Brezzi-Douglas-Marini $(BDM)$ finite element, the well-posedness of $\boldsymbol z$ is obvious. 
From a simple scaling argument, the estimate is desired.
\end{proof}

Now we consider the inf-sup condition of $b_h(\boldsymbol q_h,\tilde u_h)$.

\begin{theorem}\label{uniform:inf-sup:bh}
For $k\ge 1$, assume that $\boldsymbol Q_h=\boldsymbol Q_h^k, V_h=V_h^{k-1}$ and $\hat V_h=\hat V_h^r$, 
where $0\le r\le k$, or $\boldsymbol Q_h=\boldsymbol Q_h^{k,RT}, V_h=V_h^{k-1}$ and 
$\hat V_h=\hat V_h^r,$ where $0\le r\le k-1$, then we have:
\begin{equation}
 \inf_{\tilde u_h\in \tilde V_h}\sup_{\boldsymbol{q}_h\in \boldsymbol{Q}_h}   \frac{b_h(\boldsymbol q_h,\tilde u_h)}{\|\boldsymbol q_h\|_{{\rm div},\rho,h}\|\tilde u_h\|_{0,\rho,h}}    \geq \beta_2,
\end{equation}
where $\beta_2>0$ is a constant independent of $\rho$ and mesh size $h$.
\end{theorem}
\begin{proof}
Here we only give the proof under the assumption $\boldsymbol Q_h=\boldsymbol Q_h^k, V_h=V_h^{k-1}$ and $\hat V_h=\hat V_h^r$. The other case is similar. 

For any $\tilde u_h\in \tilde V_h$, namely for any $u_h\in V_h, \hat u_h\in \hat V_h$, we need to 
construct a $\boldsymbol q_h\in \boldsymbol Q_h$,  such that: 
\begin{equation}
b_h(\boldsymbol q_h,\tilde u_h)=\|\tilde u_h\|^2_{0,\rho,h}~~\hbox{and}~~
\|\boldsymbol q_h\|_{{\rm div},\rho,h}\lesssim \|\tilde u_h\|_{0,\rho,h}.
\end{equation}
We define $\boldsymbol z_h$ piecewisely on any $K$, namely $\boldsymbol z_h\in Q_h, \boldsymbol z_h|_K=\boldsymbol z_K$ and $\boldsymbol z_K\in \boldsymbol{\mathcal P_r(K)}$ is defined as follows 
\begin{eqnarray}
&&(\boldsymbol z_K, \boldsymbol p)_K=0, ~~\forall~ \boldsymbol p\in \boldsymbol{\mathcal P_{r-2}(K)} \oplus \boldsymbol{\mathcal  S_{r-1}(K)},\\
&&\langle\boldsymbol z_K\cdot \boldsymbol n_i-\frac{\rho h_{e_i}\hat u_h}{2}, \hat v\rangle_{e_i}=0,~~\forall~ \hat v\in \mathcal P_r(e_i), i=1,\cdots, d+1.
\end{eqnarray}
Then, by Lemma \ref{BDM:jump}, we have: 
\begin{equation}\label{bound:Z_h}
\|\boldsymbol z_h\|^2\lesssim \rho^2 h_K h^2_e\sum_{e\in \mathcal{E}_h^i}\|\hat u_h\|^2_{0,e}.
\end{equation}
In fact, we also have that for any $e\in \mathcal{E}_h^i$,
\begin{equation} \label{eq:z_hjump}
[\boldsymbol z_h]|_e=\rho h_e\hat u_h|_e.
\end{equation}
Next, noting that ${\rm div} \boldsymbol z_K\in \mathcal P_{r-1}(K)$, then for $-div_h\boldsymbol z_h-u_h$, 
there exists $\boldsymbol r_h\in H({\rm div}, \Omega)\cap \boldsymbol Q_h$ such that:
\begin{eqnarray}\label{bound:r_h}
&&{\rm div} \boldsymbol r_h=-{\rm div}_h\boldsymbol z_h-u_h,\\
&&\|\boldsymbol r_h\|+\|{\rm div} \boldsymbol r_h\|\lesssim \|-{\rm div}_h\boldsymbol z_h-u_h\|.
\end{eqnarray}
Now we define $\boldsymbol q_h=\boldsymbol z_h+\boldsymbol r_h$, noting that 
$\boldsymbol r_h\in H({\rm div}, \Omega)\cap \boldsymbol Q_h $, namely for any 
$e\in \mathcal{E}_h^i, [\boldsymbol r_h]|_e=0$; hence, for any $e\in \mathcal{E}_h^i$:
\begin{equation}\label{eq:jumpq_h}
[\boldsymbol q_h]|_e=[\boldsymbol r_h] |_e+[\boldsymbol z_h]|_e=[\boldsymbol z_h]|_e=\rho h_e\hat u_h|_e 
\end{equation}
and 
\begin{equation}\label{eq:div_hq_h}
{\rm div}_h \boldsymbol q_h={\rm div}_h (\boldsymbol z_h+\boldsymbol r_h)=-u_h.
\end{equation}
Substituting \eqref{eq:jumpq_h} and \eqref{eq:div_hq_h} into $b_h(\boldsymbol q_h,\tilde u_h)$, 
we immediately obtain:
$$
b_h(\boldsymbol q_h,\tilde u_h)=-(u_h, {\rm div}_h \boldsymbol{q}_h)_{\mathcal T_h}+\sum\limits_{e\in \mathcal{E}_h^i}\langle\hat{u}_h,[\boldsymbol{q}_h]\rangle_e=\|\tilde u_h\|^2_{0,\rho,h}.
$$
Finally, by \eqref{bound:Z_h}, \eqref{bound:r_h}, inverse inequality, and \eqref{eq:z_hjump}, 
noting that for any $e\in \mathcal{E}_h^i, [\boldsymbol r_h]|_e=0$, we have:
\begin{eqnarray}
\nonumber\|\boldsymbol q_h\|_{{\rm div},\rho,h}&=&\|\boldsymbol r_h+\boldsymbol z_h\|_{\rm Div} \leq\|\boldsymbol r_h\|_{\rm div}+\|\boldsymbol z_h\|_{\rm div}\leq\|\boldsymbol r_h\|+\|{\rm div} \boldsymbol r_h\|+\|\boldsymbol z_h\|_{\rm div}\\
&\lesssim&  \|-{\rm div}_h\boldsymbol z_h-u_h\|+\|\boldsymbol z_h\|+ \|{\rm div}_h\boldsymbol z_h\|+\Big(\rho^{-1}\sum_{e\in \mathcal {E}_h^i}h^{-1}_e\langle\hat P_e ([\boldsymbol z_h]),\hat P_e ([\boldsymbol z_h])\rangle_e\Big)^{1/2}\\
&\lesssim& \|u_h\|+(1+h^{-1})\|\boldsymbol z_h\|+\Big(\rho^{-1}\sum_{e\in \mathcal {E}_h^i}h^{-1}_e\langle[\boldsymbol z_h],[\boldsymbol z_h]\rangle_e\Big)^{1/2}\\
&=& \|u_h\|+(1+h^{-1})\Big( h\sum_{e\in \mathcal{E}_h^i}\rho^2h^2_e\|\hat u_h\|^2_{0,e}\Big)^{1/2}+\Big(\rho^{-1}\sum_{e\in \mathcal {E}_h^i}h^{-1}_e\langle\rho h_e \hat u_h,\rho h_e \hat u_h\rangle_e\Big)^{1/2}\\
&\lesssim& \|u_h\|+\Big(\rho\sum_{e\in \mathcal {E}_h^i}h_e\langle \hat u_h, \hat u_h\rangle_e\Big)^{1/2}\lesssim \|\tilde u_h\|_{0,\rho,h}.
\end{eqnarray}
Therefore, we obtain the proof.
\end{proof}
\begin{remark}
In fact, we also can choose $\tau=0$ in \eqref{definition:ahbhch} and for any $\tilde v\in \tilde V_h$, we define:
$$
\|\tilde v_h\|_{0, \rho,h}^2=(v_h,v_h)_{\mathcal T_h}+\sum\limits_{e\in \mathcal{E}^i_h}h_e\langle\hat v_h,\hat v_h\rangle_e.
$$
Define norms for $\boldsymbol{p}_h\in \boldsymbol{Q}_h$  as follows:
$$
\|\boldsymbol{p}_h\|_{{\rm div},\rho,h}^2=(c \boldsymbol{p}_h, \boldsymbol{p}_h)_{\mathcal T_h}+( {\rm div}\boldsymbol{p}_h, {\rm div} \boldsymbol{p}_h)_{\mathcal T_h}+\sum\limits_{e\in \mathcal{E}^i_h}h_e^{-1} \langle\hat P_e([\boldsymbol{p}_h]),\hat P_e([\boldsymbol{p}_h])\rangle_e,
$$
where $\hat P_e: L^2(e)\rightarrow \hat V(e)$ still is the $L^2$ projection.
Then, we can get the stability result by a similar proof.
\end{remark}

\subsection{Proof for Part  \ref{wellposed:Ah:grad} of Theorem \ref{wellposed:Ah:divgrad}  }
Next, we prove part \ref{wellposed:Ah:grad} of Theorem \ref{wellposed:Ah:divgrad}.
The uniform boundedness of $A_h((\boldsymbol{p}_h, \tilde u_h),(\boldsymbol{q}_h, \tilde v_h))$ is obvious. 
The uniform inf-sup condition for $A_h((\boldsymbol{p}_h, \tilde u_h),(\boldsymbol{q}_h, \tilde v_h))$ is as follows:
\begin{theorem}
Assume $\nabla_h V_h \subset \boldsymbol Q_h$, then there exists a positive constant $\rho_0$ that 
only depends on the shape regularity of the mesh, such that for any $0<\rho \leq\rho_0$, we have:
\begin{equation}
 \inf_{(\boldsymbol{p}_h, \tilde u_h)\in \boldsymbol{Q}_h\times \tilde V_h}\sup_{\boldsymbol{q}_h, \tilde v_h\in \boldsymbol{Q}_h\times \tilde V_h}   \frac{A_h((\boldsymbol{p}_h, \tilde u_h),(\boldsymbol{q}_h, \tilde v_h))}{(\|{\tilde u}_h\|_{\tilde 1,\rho,h}+\|\boldsymbol{p}_h\|)(\|{\tilde v}_h\|_{\tilde 1,\rho,h}+\|\boldsymbol{q}_h\|)}    \geq \beta_3,
\end{equation}
where $\beta_3>0$ is a constant independent of $\rho$ and mesh size $h$.
\end{theorem}
\begin{proof}
For any given $(\boldsymbol{p}_h, \tilde u_h)\in \boldsymbol{Q}_h\times \tilde V_h$, since 
$\nabla_h V_h \subset \boldsymbol Q_h$, we choose $\boldsymbol{q}_{h}=\boldsymbol{p}_{h}+\nabla_h u_h$ 
and $\tilde v_h=-\tilde u_h$,
and then we have the following boundedness of $\boldsymbol q_h$ and $\tilde v_h$ by $\boldsymbol q_h$ and $\tilde v_h$
\begin{equation}\label{boundedqh}
\|\boldsymbol{q}_{h}\|^2=(c(\boldsymbol{p}_{h}+\nabla_h u_h),\boldsymbol{p}_{h}+\nabla_h u_h)_{\mathcal T_h}\leq 2\big(\|\boldsymbol{p}_{h}\|^2+\|\nabla_h u_h\|^2\big),
\end{equation}
and
\begin{equation}
\|\tilde{v}_h\|_{\tilde 1,\rho,h}=\|-\tilde{u}_h\|_{\tilde 1,\rho,h}=\|\tilde{u}_h\|_{\tilde 1,\rho,h}.
\end{equation}
On the other hand, 
\begin{equation*}
\begin{split}
&A_h((\boldsymbol{p}_h, \tilde u_h),(\boldsymbol{q}_h, \tilde v_h))\\
&=(c\boldsymbol{p}_h, \boldsymbol{q}_h)_{\mathcal T_h}+(\nabla_h u_h, \boldsymbol{q}_h)_{\mathcal T_h}-\langle u_h-\hat{u}_h, \boldsymbol{q}_h\cdot\boldsymbol{n}\rangle_{\partial {\mathcal T_h}}\\
&+(\nabla_h v_h, \boldsymbol{p}_h)_{\mathcal T_h}-\langle v_h-\hat{v}_h, \boldsymbol{p}_h\cdot\boldsymbol{n}\rangle_{\partial {\mathcal T_h}}
-\rho^{-1}\sum\limits_{K\in \mathcal{T}_h}h_K^{-1}\langle u_h-\hat{u}_h, v_h-\hat{v}_h\rangle_{\partial K}\\
&=(c\boldsymbol{p}_h, \boldsymbol{p}_h+\nabla_h u_h)_{\mathcal T_h}+(\nabla_h u_h, \boldsymbol{p}_h+\nabla_h u_h)_{\mathcal T_h}-\langle u_h-\hat{u}_h, \boldsymbol{q}_h\cdot\boldsymbol{n}\rangle_{\partial {\mathcal T_h}}\\
&+(-\nabla_h u_h, \boldsymbol{p}_h)_{\mathcal T_h}-\langle -(u_h-\hat{u}_h), \boldsymbol{p}_h\cdot\boldsymbol{n}\rangle_{\partial {\mathcal T_h}}
+\rho^{-1}\sum\limits_{K\in \mathcal{T}_h}h_K^{-1}\langle u_h-\hat{u}_h, u_h-\hat{u}_h\rangle_{\partial K}\\
&=(c\boldsymbol{p}_h, \boldsymbol{p}_h)_{\mathcal T_h}+(c\boldsymbol{p}_h, \nabla_h u_h)_{\mathcal T_h}+(\nabla_h u_h, \nabla_h u_h)_{\mathcal T_h}
+\sum\limits_{K\in \mathcal{T}_h}\langle (u_h-\hat{u}_h), (\boldsymbol{p}_h-\boldsymbol{q}_h)\cdot\boldsymbol{n}_K\rangle_{\partial K}\\
&+\rho^{-1}\sum\limits_{K\in \mathcal{T}_h}h_K^{-1}\langle u_h-\hat{u}_h, u_h-\hat{u}_h\rangle_{\partial K}\\
&\geq \frac{1}{2} \big(\|\boldsymbol{p}_h\|^2+\|\nabla_h u_h\|^2\big)-\epsilon C_5\|\boldsymbol{p}_h-\boldsymbol{q}_h\|^2+(\rho^{-1}-\epsilon^{-1}) \sum\limits_{K\in \mathcal{T}_h}h_K^{-1}\langle u_h-\hat{u}_h, u_h-\hat{u}_h\rangle_{\partial K}\\
&= \frac{1}{2} \big(\|\boldsymbol{p}_h\|^2+\|\nabla_h u_h\|^2\big)-\epsilon C_5\|\nabla_h u_h\|^2+(\rho^{-1}-\epsilon^{-1}) \sum\limits_{K\in \mathcal{T}_h}h_K^{-1}\langle u_h-\hat{u}_h, u_h-\hat{u}_h\rangle_{\partial K}.
\end{split}
\end{equation*}
where $C_5$ is a constant independent of $\rho$ and $h$.

Now setting $\epsilon=\frac{1}{4C_5}, \rho_0= \frac{3}{16C_5}$, then for any $\rho\leq \rho_0$, we have:
\begin{equation*}
\begin{split}
A_h((\boldsymbol{p}_h, \tilde u_h),(\boldsymbol{q}_h, \tilde v_h))
&\geq \frac{1}{2} \|\boldsymbol{p}_h\|^2+(\frac{1}{2}-\epsilon C_5)\|\nabla_h u_h\|^2+\rho^{-1}(1-\rho\epsilon^{-1}) \sum\limits_{K\in \mathcal{T}_h}h_K^{-1}\langle u_h-\hat{u}_h, u_h-\hat{u}_h\rangle_{\partial K}\\
&\geq \frac{1}{2} \|\boldsymbol{p}_h\|^2+\frac{1}{4}\|\nabla_h u_h\|^2+\frac{1}{4} \rho^{-1}\sum\limits_{K\in \mathcal{T}_h}h_K^{-1}\langle u_h-\hat{u}_h, u_h-\hat{u}_h\rangle_{\partial K}\\
&\geq \frac{1}{4} \big(\|\boldsymbol{p}_h\|^2+\|\tilde u_h\|^2_{\tilde 1,\rho,h}\big).
\end{split}
\end{equation*}
Hereby, we complete the proof.
\end{proof}

\subsection{Proof for Part \ref{wellposed:Aw:grad} of Theorem \ref{wellposed:Aw:graddiv}}

By the definition of the norms, the continuity and coercivity of $a_w(\cdot,\cdot)$ is obvious, namely, 
\begin{theorem}\label{bounded:coercivity:aw}
For any $0<\rho\le 1$, we have:
$$
|a_w(\boldsymbol{\tilde p}_h,\boldsymbol{\tilde q}_h)| \leq \|\boldsymbol{\tilde p}_h\|_{0,h,\rho} \|\boldsymbol{\tilde q}_h\|_{0,h,\rho}~\forall~~  \boldsymbol{\tilde p}_h\in \boldsymbol{\tilde Q}_h,\boldsymbol{\tilde q}_h\in \boldsymbol{\tilde Q}_h.
$$
$$
a_w(\boldsymbol{\tilde p}_h,\boldsymbol{\tilde p}_h)\geq\|\boldsymbol{\tilde p}_h\|^2_{0,h,\rho}~~ \forall  ~~\boldsymbol{\tilde p}_h\in \boldsymbol{\tilde Q}_h.
$$
\end{theorem}
Before we prove the boundedness and inf-sup condition of $b_w(\boldsymbol{\tilde p}_h,v_h)$, by identity \eqref{equ:dg-identity_1} and noting that $[\boldsymbol{\hat p}_h]=0$, we rewrite $b_w(\boldsymbol{\tilde p}_h,v_h)$ as:
\begin{equation}\label{bw_identity}
b_w(\boldsymbol{\tilde p}_h,v_h)=
(\boldsymbol{p}_h, \nabla_h v_h)_{\mathcal T_h}-(\boldsymbol{\hat p}_h\cdot\boldsymbol{n}_K,{v}_h)_{\partial \mathcal T_h}=(\boldsymbol{p}_h, \nabla v_h)_{\mathcal T_h}-\langle \boldsymbol{\hat p}_h, \lbrack\!\lbrack v_h\rbrack\!\rbrack\rangle_{\mathcal E_h}.
\end{equation} 
Then, the boundedness of $b_w(\boldsymbol{\tilde p}_h,v_h)$ is as follows:
\begin{theorem}\label{bounded:bw}
For any $0<\rho\leq 1$, and for any $\boldsymbol{\tilde p}_h\in \boldsymbol{\tilde Q}_h, v_h\in V_h$,
\begin{equation}
b_w(\boldsymbol{\tilde p}_h,v_h)\leq C_w \|\boldsymbol{\tilde p}_h\|_{0,h\rho} \|v_h\|_{1,h,\rho}.
\end{equation}
\end{theorem}
\begin{proof}
Using the Cauchy inequality for \eqref{bw_identity}, we obtain:
\begin{equation}\label{bounded_bw}
\begin{split}
|b_w(\boldsymbol{\tilde p}_h,v_h)|&\leq\|\boldsymbol{p}_h\|\|\nabla_h v_h\|+\sum\limits_{e\in \mathcal{E}_h}\| \boldsymbol{\hat p}_h\cdot \boldsymbol n_e\|_{0,e}\|\hat Q_e([v_h])\|_{0,e}\\
& \leq\|\boldsymbol{p}_h\|\|\nabla_h v_h\|+\Big(\rho\sum\limits_{e\in \mathcal{E}_h}h_e\| \boldsymbol{\hat p}_h\cdot \boldsymbol n_e\|^2_{0,e}\Big)^{1/2} \Big(\rho^{-1}\sum\limits_{e\in \mathcal{E}_h}h_e^{-1}\|\hat Q_e([v_h])\|^2_{0,e}\Big)^{1/2}\\
& \leq\Big(\|\boldsymbol{p}_h\|+\Big(\rho\sum\limits_{e\in \mathcal{E}_h}h_e\| \boldsymbol{\hat p}_h\cdot \boldsymbol n_e\|^2_{0,e}\Big)^{1/2}\Big) \|v_h\|_{1,h,\rho}.
\end{split}
\end{equation}
Let $K$ be an element that takes e as an edge or flat face. Then, using the trace
inequality and the inverse inequality we obtain:
\begin{equation}\label{bounded_bw_1}
h_e\| \boldsymbol{\hat p}_h\cdot \boldsymbol n_e\|^2_{0,e}\leq 2 h_e \| (\boldsymbol{\hat p}_h-\boldsymbol p_h)\cdot \boldsymbol n_e\|^2_{0,e}+2h_e \|\boldsymbol p_h\cdot \boldsymbol n_e\|_{0,e}^2\leq C_t (h_e \| (\boldsymbol{\hat p}_h-\boldsymbol p_h)\cdot \boldsymbol n_e\|^2_{0,e}+ \|\boldsymbol p_h\|_{0,K}^2).
\end{equation}
Substituting the above inequality \eqref{bounded_bw_1} into \eqref{bounded_bw} yields:
$$
b_w(\boldsymbol{\tilde p}_h,v_h)\leq C_w \|\boldsymbol{\tilde p}_h\|_{0,h,\rho} \|v_h\|_{1,h,\rho}.
$$
Hence, the lemma is proved.
\end{proof}

We also have the following uniform inf-sup condition for $b_w(\boldsymbol{\tilde p}_h,v_h)$:
\begin{theorem}\label{inf-sup:bw}
Assume $\nabla_h V_h\subset \boldsymbol Q_h$, then for any $0<\rho\leq 1$, we have: 
\begin{equation}
 \inf_{v_h\in V_h}\sup_{\boldsymbol{\tilde p}_h\in\boldsymbol{\tilde Q}_h}   \frac{b_w(\boldsymbol{\tilde p}_h,v_h)}{\|{v}_h\|_{1,h,\rho}\|\boldsymbol{\tilde p}_h\|_{0,h,\rho}}    \geq \beta_4,
\end{equation}
where $\beta_4>0$  is independent of mesh size $h$ and $\rho$.
\end{theorem}
\begin{proof}
Since $\nabla_h V_h\subset \boldsymbol Q_h$, taking $\boldsymbol{p}_h=\nabla_h v_h, \boldsymbol{\hat p}_h=-\rho^{-1} h_e^{-1}\hat Q_e([v_h])\boldsymbol{n}_e$ in \eqref{bw_identity}, we have 
$$
b_w(\boldsymbol{\tilde p}_h,v_h)=(\nabla_h v_h, \nabla_h v_h)_{\mathcal T_h}+\rho^{-1} \sum\limits_{e\in \mathcal{E}_h}h_e^{-1}\langle \hat Q_e([v_h]),\hat Q_e([v_h])\rangle_{e}
=\|v_h\|_{1,h,\rho}^2.
$$
Noting that $\rho\leq 1$, we obtain:
\begin{equation}
\begin{split}
\|\boldsymbol{\tilde p}_h\|_{0,h,\rho}^2&=(c \nabla_h v_h, \nabla_h v_h)_{\mathcal T_h}+\rho\sum\limits_{K\in \mathcal{T}_h}h_K\|\nabla_h v_h\cdot \boldsymbol{n_K}+\rho^{-1} h_e^{-1}\hat Q_e([v_h]) \boldsymbol n_e\cdot \boldsymbol n_K\|_{0,\partial K}^2 \\
&\leq \beta_4 \big((c \nabla_h v_h, \nabla_h v_h)_{\mathcal T_h}+\rho^{-1}\sum\limits_{e\in \mathcal{E}_h}h_e^{-1}\|\hat Q_e([v_h])\|_{0,e}^2 \big)\\
&\leq \beta_4 \|v_h\|_{1,h,\rho}^2.
\end{split}
\end{equation}
Here, we obtain the desired result.
\end{proof}

\subsection{Proof for Part \ref{wellposed:Aw:div} of Theorem \ref{wellposed:Aw:graddiv}}
Next, we prove part \ref{wellposed:Aw:div} of Theorem \ref{wellposed:Aw:div}.
The uniform boundedness of $A_w((\cdot, \cdot),(\cdot, \cdot))$ is obvious. The uniform inf-sup of 
$A_w((\cdot, \cdot),(\cdot, \cdot))$ is as follows:
\begin{theorem}
Let $\boldsymbol R_h\subset H({\rm div}, \Omega)\cap
\boldsymbol Q_h$ be the Raviart-Thomas finite element space. Assume
that $\{\!\!\{\boldsymbol R_h\}\!\!\}\subset \hat Q_h$ and
$V_h = {\rm div}_h \boldsymbol Q_h$. Then, for $0<\rho\le 1$ the bilinear 
form $A_w((\cdot, \cdot),(\cdot, \cdot))$ with $\eta=\rho^{-1} h^{-1}_K$ satisfies:
\begin{equation} \label{equ:WG-infsup-div}
 \inf_{(\tilde{\boldsymbol p}_h, u_h)\in \tilde{\boldsymbol Q}_h\times V_h} 
 \sup_{(\tilde{\boldsymbol q}_h, v_h)\in \tilde{\boldsymbol
Q}_h\times V_h}   \frac{A_w((\tilde{\boldsymbol p}_h,
u_h),(\tilde{\boldsymbol q}_h, v_h))}{(\|u_h\|+\|\tilde{\boldsymbol
p}_h\|_{\widetilde{\rm div},\rho,h})(\|v_h\|+\|\tilde{\boldsymbol
  q}_h\|_{\widetilde{\rm div},\rho,h})} \geq \beta_5,
\end{equation}
where $\beta_5>0$ is a constant independent of both $\rho$ and mesh size $h$.
\end{theorem}
\begin{proof}
For any given $(\boldsymbol{\tilde p}_h, u_h)\in \boldsymbol{ \tilde Q}_h\times  V_h$, 
namely $(\boldsymbol{p}_h, \boldsymbol {\hat p}_h, u_h)\in \boldsymbol{Q}_h\times\hat Q_h \times V_h$.
Since $V_h\subset {\rm div} \boldsymbol R_h$ and $\boldsymbol R_h\times V_h$ such that the 
mixed conforming method is well-defined, there exists $\boldsymbol r_h\in \boldsymbol R_h$ such that:
\begin{equation}\label{RT0}
-{\rm div} \boldsymbol r_h=u_h~~ \hbox{and}~~ \|\boldsymbol r_h\|+\|{\rm div} \boldsymbol r_h\|\leq C \|u_h\|.
\end{equation}
Now we choose $\boldsymbol q_h=\boldsymbol r_{h}+\alpha\boldsymbol p_{h}, 
\boldsymbol {\hat q}_h= \alpha\boldsymbol {\hat p}_h+ 
(\boldsymbol  {r}_{h}\cdot \boldsymbol n_e)\boldsymbol n_e, v_h=-{\rm div}_h \boldsymbol p_h-\alpha u_h$, 
where $\alpha$ is a constant that will be indicated later. 

We then first verify the boundedness of $(\boldsymbol {\tilde q}_h,v_h)$ by $(\boldsymbol {\tilde p}_h,u_h)$. 

Noting that $(\boldsymbol q_h- \boldsymbol {\hat q}_h)\cdot \boldsymbol n_K|_{\partial K}=(\boldsymbol r_{h}+\alpha\boldsymbol p_{h}-\alpha\boldsymbol {\hat p}_h- (\boldsymbol  {r}_{h}\cdot \boldsymbol n_e)\boldsymbol n_e)\cdot \boldsymbol n_K|_{\partial K}=\alpha (\boldsymbol p_h- \boldsymbol {\hat p}_h)\cdot \boldsymbol n_K|_{\partial K}$, we have: 
\begin{equation}
\begin{split}
\|\boldsymbol {\tilde q}_h\|_{\widetilde{\rm div},\rho,K}^2&=(c \boldsymbol{q}_h, \boldsymbol{q}_h)_K+({\rm div}\boldsymbol{q}_h, {\rm div}\boldsymbol{q}_h)_K+\rho^{-1} h^{-1}_K\langle(\boldsymbol{q}_h-\boldsymbol{\hat q}_h)\cdot \boldsymbol{n}_K, (\boldsymbol{q}_h-\boldsymbol{\hat q}_h)\cdot \boldsymbol{n}_K\rangle_{\partial K}\\
&=(c(\boldsymbol r_{h}+\alpha\boldsymbol p_{h}), \boldsymbol r_{h}+\alpha \boldsymbol p_{h})_K+({\rm div}\boldsymbol r_{h}+\alpha {\rm div}\boldsymbol p_{h}, {\rm div}\boldsymbol r_{h}+\alpha {\rm div}\boldsymbol p_{h})_K\\
&+\alpha^2\rho^{-1}h^{-1}_K\langle(\boldsymbol{p}_h-\boldsymbol{\hat p}_h)\cdot \boldsymbol{n}_K, (\boldsymbol{p}_h-\boldsymbol{\hat p}_h)\cdot \boldsymbol{n}_K\rangle_{\partial K}\\
&\leq 2\|\boldsymbol r_{h}\|^2+2\alpha^2\|\boldsymbol p_{h}\|^2+2\|{\rm div} \boldsymbol r_{h}\|^2+({\rm div}_h \boldsymbol p_{h},{\rm div}_h \boldsymbol p_{h})\\
&+\alpha^2\rho^{-1} h^{-1}_K\langle(\boldsymbol{p}_h-\boldsymbol{\hat p}_h)\cdot \boldsymbol{n}_K, (\boldsymbol{p}_h-\boldsymbol{\hat p}_h)\cdot \boldsymbol{n}_K\rangle_{\partial K}\\
&\leq 2(C^2\|u_h\|_{0,K}^2+\alpha^2\|\boldsymbol {\tilde p}_h\|^2_{\widetilde{\rm div},\rho,K}).
\end{split}
\end{equation}
Hence, $\|\boldsymbol {\tilde q}_h\|_{\widetilde{\rm div},h,\rho}^2\leq 2(C^2\|u_h\|^2
+\alpha^2\|\boldsymbol {\tilde p}_h\|^2_{\widetilde{\rm div},h,\rho})$. 
Further,
\begin{equation}
\begin{split}
\|v_h\|=\|-{\rm div}_h \boldsymbol p_h-\alpha u_h\|=\|{\rm div}_h \boldsymbol p_h \|+\alpha\|u_h\|\leq \|\boldsymbol {\tilde p}_h\|_{\widetilde{\rm div},h,\rho}+\alpha\|u_h\|.
\end{split}
\end{equation}
Then, we prove the boundedness of $(\boldsymbol {\tilde q}_h,v_h)$ by $(\boldsymbol {\tilde p}_h,u_h)$.

Now through integration by parts, we have the following:
\begin{eqnarray}
&&A_w((\boldsymbol{ \tilde p}_h, u_h),(\boldsymbol{ \tilde q}_h, v_h))\\
&=&(c \boldsymbol{p}_h, \boldsymbol{q}_h)_{\mathcal T_h}+\rho^{-1}\sum\limits_{K\in \mathcal{T}_h}h^{-1}_K\langle(\boldsymbol{p}_h-\boldsymbol{\hat p}_h)\cdot \boldsymbol{n}_K, (\boldsymbol{q}_h-\boldsymbol{\hat q}_h)\cdot \boldsymbol{n}_K\rangle_{\partial K}\\
&&\nonumber+(\boldsymbol{q}_h, \nabla_h u_h)_{\mathcal T_h}-\langle\boldsymbol{\hat q}_h\cdot\boldsymbol{n},{u}_h\rangle_{\partial {\mathcal T_h}}+(\boldsymbol{p}_h, \nabla_h v_h)_{\mathcal T_h}-\langle\boldsymbol{\hat p}_h\cdot\boldsymbol{n},{v}_h\rangle_{\partial {\mathcal T_h}}\\
&=&\nonumber(c \boldsymbol{p}_h, \boldsymbol{q}_h)_{\mathcal T_h}+\rho^{-1}\sum\limits_{K\in \mathcal{T}_h}h^{-1}_K\langle(\boldsymbol{p}_h-\boldsymbol{\hat p}_h)\cdot \boldsymbol{n}_K, (\boldsymbol{q}_h-\boldsymbol{\hat q}_h)\cdot \boldsymbol{n}_K\rangle_{\partial K}-({\rm div}_h \boldsymbol{q}_h, u_h)_{\mathcal T_h}\\
&&\nonumber+\langle(\boldsymbol{q}_h-\boldsymbol{\hat q}_h)\cdot\boldsymbol{n},{u}_h\rangle_{\partial {\mathcal T_h}}-({\rm div}_h \boldsymbol{p}_h, v_h)_{\mathcal T_h}+\langle(\boldsymbol{p}_h-\boldsymbol{\hat p}_h)\cdot\boldsymbol{n},{v}_h\rangle_{\partial {\mathcal T_h}}
\end{eqnarray}
By the Cauchy inequality and inverse inequality, we have:
\begin{eqnarray}
\lefteqn{A_{w,K}((\boldsymbol{ \tilde p}_h, u_h),(\boldsymbol{ \tilde q}_h, v_h))}\\
&=&
(c \boldsymbol{p}_h,
\boldsymbol{r}_h+\alpha\boldsymbol{p}_h)_K+\alpha\rho^{-1} h^{-1}_K\langle(\boldsymbol{p}_h-\boldsymbol{\hat p}_h)\cdot \boldsymbol{n}_K, (\boldsymbol{p}_h-\boldsymbol{\hat p}_h)\cdot \boldsymbol{n}_K\rangle_{\partial K}\\
&&-
({\rm div} \boldsymbol{r}_h+\alpha {\rm div} \boldsymbol{p}_h, u_h)_K+\alpha\langle(\boldsymbol{p}_h-\boldsymbol{\hat p}_h)\cdot\boldsymbol{n}_K,{u}_h\rangle_{\partial K}\\
&&-({\rm div} \boldsymbol{p}_h, -{\rm div}\boldsymbol{p}_h -\alpha u_h)_K+  \langle(\boldsymbol{p}_h-\boldsymbol{\hat p}_h)\cdot\boldsymbol{n}_K,-{\rm div}\boldsymbol{p}_h -\alpha u_h\rangle_{\partial K}\\
&=&(c \boldsymbol{p}_h, \boldsymbol{r}_h)_K+\alpha  (c \boldsymbol{p}_h,\boldsymbol{p}_h)_K+\alpha\rho^{-1} h^{-1}_K\langle(\boldsymbol{p}_h-\boldsymbol{\hat p}_h)\cdot \boldsymbol{n}_K, (\boldsymbol{p}_h-\boldsymbol{\hat p}_h)\cdot \boldsymbol{n}_K\rangle_{\partial K}\\
&&+(u_h, u_h)_K
+  ({\rm div} \boldsymbol{p}_h, {\rm div}\boldsymbol{p}_h)_K-  \langle(\boldsymbol{p}_h-\boldsymbol{\hat p}_h)\cdot\boldsymbol{n}_K,{\rm div}\boldsymbol{p}_h \rangle_{\partial K}\\
&\geq& -\epsilon_1 \|\boldsymbol{r}_h\|_{0,K}^2-\epsilon_1^{-1} \|\boldsymbol{p}_h\|_{0,K}^2+ \alpha \|\boldsymbol{p}_h\|_{0,K}^2+\alpha\rho^{-1}  h^{-1}_K\langle(\boldsymbol{p}_h-\boldsymbol{\hat p}_h)\cdot \boldsymbol{n}_K, (\boldsymbol{p}_h-\boldsymbol{\hat p}_h)\cdot \boldsymbol{n}_K\rangle_{\partial K}\\
&&+\|u_h\|_{0,K}^2
+  ({\rm div} \boldsymbol{p}_h, {\rm div}\boldsymbol{p}_h)_K-\epsilon_2^{-1}  h^{-1}_K\langle(\boldsymbol{p}_h-\boldsymbol{\hat p}_h)\cdot \boldsymbol{n}_K, (\boldsymbol{p}_h-\boldsymbol{\hat p}_h)\cdot \boldsymbol{n}_K\rangle_{\partial K}\\
&&-\epsilon_2  h_K({\rm div} \boldsymbol{p}_h, {\rm div}\boldsymbol{p}_h)_{\partial K}\\
&=& -\epsilon_1 \|\boldsymbol{r}_h\|_{0,K}^2+(\alpha-\epsilon_1^{-1}) \|\boldsymbol{p}_h\|_{0,K}^2+\|u_h\|_{0,K}^2+(1-\epsilon_2  C_6)({\rm div} \boldsymbol{p}_h, {\rm div}\boldsymbol{p}_h)_{K}\\
&&+(\alpha \rho^{-1}-\epsilon_2^{-1}) h^{-1}_K\langle(\boldsymbol{p}_h-\boldsymbol{\hat p}_h)\cdot \boldsymbol{n}_K, (\boldsymbol{p}_h-\boldsymbol{\hat p}_h)\cdot \boldsymbol{n}_K\rangle_{\partial K}
\end{eqnarray}
Noting \eqref{RT0}, we have the following:
\begin{eqnarray}
A_w((\boldsymbol{ \tilde p}_h, u_h),(\boldsymbol{ \tilde q}_h, v_h))
&\geq&(1-\epsilon_1 C) \|u_h\|^2+(\alpha- \epsilon_1^{-1}) \|\boldsymbol{p}_h\|^2+(1-\epsilon_2C_6)  ({\rm div}_h \boldsymbol{p}_h, {\rm div}_h\boldsymbol{p}_h)\\
&&+(\alpha\rho^{-1}-\epsilon_2^{-1})\sum\limits_{K\in \mathcal{T}_h}  h^{-1}_K\langle(\boldsymbol{p}_h-\boldsymbol{\hat p}_h)\cdot \boldsymbol{n}_K, (\boldsymbol{p}_h-\boldsymbol{\hat p}_h)\cdot \boldsymbol{n}_K\rangle_{\partial K} .
\end{eqnarray}

Now choosing $\epsilon_1=\frac{1}{2C}, \epsilon_2=\frac{1}{2C_6}, \alpha=\max\{2C+\frac{1}{2},2C_6+\frac{1}{2}\}, 
0<\rho\leq 1$, we have:
\begin{eqnarray}
A_w((\boldsymbol{ \tilde p}_h, u_h),(\boldsymbol{ \tilde q}_h, v_h))&\geq& \frac{1}{2}\Big(\|\boldsymbol{p}_h\|^2+\sum\limits_{K\in \mathcal{T}_h}h^{-1}_K\|(\boldsymbol{p}_h-\boldsymbol{\hat p}_h)\cdot \boldsymbol{n}_K\|^2_{0,\partial K}+ \|u_h\|^2+\|{\rm div}_h \boldsymbol{p}_h\|^2\Big)\\
& \geq& \frac{1}{2}\Big(\|u_h\|^2+\|\boldsymbol{\tilde p}_h\|_{\widetilde{\rm div},h,\rho}\Big).
\end{eqnarray}
Thus, we prove the theorem.
\end{proof}

\section{Summary} \label{sec:concluding}
In this paper we use the classic LBB  theory to prove two types of uniform stability results under some proper
parameter-dependent norms for HDG methods, which are uniformly stable with respect to the 
stabilization parameters and mesh size $h$. Based on the uniform stability results, 
we further prove uniform and optimal error estimates for HDG methods, which are 
independent of the stabilization parameters. In addition, we also prove two types of uniform stability results for 
WG methods. Similarly based on the uniform stability results, we further prove uniform and optimal error estimates for WG methods. These uniform stability results and optimal error estimates for WG methods are meaningful. Following these uniform stability results for HDG 
methods and WG methods presented in this paper, an HDG method is shown to converge to a 
primal conforming method, whereas a WG method is shown to converge to a mixed conforming method by taking the 
limit of the stabilization parameters.



\bibliographystyle{unsrt}
\bibliography{mixedDG}

\end{document}